\documentclass[sn-mathphys,pdflatex]{sn-jnl}% Math and Physical Sciences Reference Style
\usepackage[T1]{fontenc}
\usepackage{mathtools}
\usepackage{amsmath}
\usepackage{empheq} 
\usepackage{color}
\usepackage{amssymb}
\usepackage{fancybox}
\usepackage{here}
\usepackage{bm}
\usepackage{enumitem}
\usepackage{subcaption}

\usepackage{ifthen}
\newcommand{\COMM}[2]{{
\ifthenelse{\equal{#1}{AT}}{\color{red}}{
\ifthenelse{\equal{#1}{PL}}{\color{blue}}}
[#1: #2]
}}

\newcommand{\No}{\mathcal{N}}
\newcommand{\Normal}{\mathcal{N}(0,1)}
\newcommand{\fgrad}{\nabla f(x_k)}

\jyear{2023}%

\theoremstyle{thmstyleone}%
\newtheorem{theorem}{Theorem}%  meant for continuous numbers
\newtheorem{proposition}[theorem]{Proposition}% 
\newtheorem{lemma}{Lemma}

\theoremstyle{thmstyletwo}%
\newtheorem{remark}{Remark}%

\theoremstyle{thmstylethree}%
\newtheorem{definition}{Definition}%
\newtheorem{assume}{Assumption}

\begin{document}

\title[Randomized subspace gradient method for constrained optimization]{Randomized subspace gradient method for constrained optimization}


\author*[1]{\fnm{Ryota} \sur{Nozawa}}\email{nozawa-ryota860@g.ecc.u-tokyo.ac.jp}

\author[2]{\fnm{Pierre-Louis} \sur{Poirion}}\email{pierre-louis.poirion@riken.jp}

\author[1,2]{\fnm{Akiko} \sur{Takeda}}\email{takeda@mist.i.u-tokyo.ac.jp}

\affil*[1]{\orgdiv{Department of Mathematical Informatics}, \orgname{The University of Tokyo}, \orgaddress{\street{Hongo}, \city{Bunkyo-ku}, \postcode{113-8656},  \state{Tokyo}, \country{Japan}}}

\affil[2]{\orgdiv{Center for Advanced Intelligence Project}, \orgname{RIKEN}, \orgaddress{\street{Nihonbashi}, \city{Chuo-ku}, \postcode{103-0027}, \state{Tokyo}, \country{Japan}}}


% \abstract{
% 	The abstract serves both as a general introduction to the topic and as a brief, non-technical summary of the main results and their implications. Authors are advised to check the author instructions for the journal they are submitting to for word limits and if structural elements like subheadings, citations, or equations are permitted.
% }
 
\abstract{ 
	We propose randomized subspace gradient methods for high-dimensional constrained optimization.
  While there have been similarly purposed studies on unconstrained optimization problems, there have been few on constrained optimization problems due to the difficulty of handling constraints.
  Our algorithms project gradient vectors onto a subspace that is a random projection of the subspace spanned by the gradients of active constraints.
	We determine the worst-case iteration complexity under linear and nonlinear settings and theoretically confirm that our algorithms can take a larger step size than their deterministic version.
	From the advantages of taking longer step and randomized subspace gradients, we show that our algorithms are especially efficient in view of time complexity when gradients cannot be obtained easily.     
	Numerical experiments show that they tend to find better solutions because of the randomness of their subspace selection.
	Furthermore, they performs well in cases where gradients could not be obtained directly, and instead, gradients are obtained using directional derivatives.
}




\maketitle

\section{Introduction}\label{sec1}

In this paper, we  consider the following constrained optimization problem:
\begin{equation}{\label{main problem}}
	\begin{split}
		&\min_{x\in \mathbb{R}^n} f(x)\\
		&\mathrm{s.t.} \; \mathcal{C} := \{ x ~ \vert ~ g_i(x)\le 0, \quad i = 1,...,m \}, 
	\end{split}
\end{equation}
where $f$ is $L$-smooth and $g_i$ are $L_g$-smooth, neither of which need be convex functions as long as an optimal solution exists.
There is a growing demand for solving large-scale constrained optimization problems in many applications such as machine learning, statistics,
and signal processing~\cite{tibshirani2005sparsity,lee1999learning,zafar2017fairness,komiyama2018icml,moldovan2012safe,achiam2017constrained};
for example, sparse optimization problems are often formulated as 
\begin{equation}
	\min_{x} L(x) \;\mathrm{s.t.} \;\mathcal{R}(x) \le s,
\label{reg_const_problem}
\end{equation}
where $L(x)$ is a loss function,  $\mathcal{R}$ is some sparsity-inducing norm such as $l_1$, and $s$ is some fixed positive integer value.
Recently, machine-learning models \eqref{reg_const_problem} including a fairness constraint~\cite{zafar2017fairness,komiyama2018icml} have attracted researchers' attention,
but difficulties have emerged in solving such large-scale problems.

For solving large-scale unconstrained problems, i.e., \eqref{reg_const_problem} with $\mathcal{C}=\mathbb{R}^n$,  \cite{kozak2021stochastic,gower2019rsn,hanzely2020stochastic,fuji2022randomized,chen2020randomized,berglund2022zeroth,cartis2022randomised} have proposed  subspace methods using random projection.
These methods update the iterate point as follows:
\begin{equation}{\label{eq:subspace gradient methods}}
	x_{k+1} = x_{k} + M_k d_k,
\end{equation}	
where $M_k$ is an $n\times d$ ($d<n$) random matrix. One of the difficulties with high-dimensional problems is calculating the gradient $\nabla f$. Although there are some methods that calculate gradients, such as automatic differentiation that is popular in machine learning, 
when the objective function has a complicated structure, backward-mode automatic differentiation leads to an explosion in memory usage~\cite{margossian2019review}.
Kozak et al.~\cite{kozak2021stochastic} proposed a stochastic subspace descent method combining gradient descent with random projections to overcome the difficulties of calculating gradients. The method uses the direction $d_k = -M_k^\top \fgrad$, 
and it was shown that $d_k$  
can be computed by using finite-difference or forward-mode automatic differentiation of the directional derivative. 
To obtain the full gradients by using finite difference or forward-mode automatic differentiation, we need to evaluate the function values $n$ times. On the other hand, the projected gradient can be computed by evaluating the function values only $d$ times.

Compared with unconstrained problems, subspace optimization algorithms using random projections have made little progress in solving constrained problems.
To the best of our knowledge, apart from \cite{cartis2022bound}, there is no paper on subspace optimization algorithms %combining optimization algorithms with random projections by 
using (\ref{eq:subspace gradient methods}) for constrained optimization \eqref{main problem}.
Here, Cartis et al.~\cite{cartis2022bound} proposed a general framework to 
investigate a general random embedding framework for bound-constrained global optimization of $f$ with low effective dimensionality. 
The framework projects the original problem onto a random subspace and  solves the reduced subproblem in each iteration:
\[
 \min_{d} f(x_{k} + M_k d) ~~ \mbox{subject to}~ x_{k} + M_k d \in \mathcal{C}.
\]
These subproblems need to be solved to some required accuracy by using a deterministic global optimization algorithm.

In this study, we propose gradient-based subspace methods for constrained optimization. The descent direction $d_k$ is obtained without solving any subproblems, by projecting the gradient vector $\nabla f$
onto a subspace that is a random projection of the subspace spanned by the gradients of the active constraints. 
A novel property of our algorithms is that when the dimension $n$ is large enough, they are almost unaffected by the constraint boundary;
because of this, they can take a longer step than their deterministic versions. 
In standard constrained optimization methods, the iterates become difficult to move when they get close to the constraint boundary.
Our methods randomly update the subspace where the iterate is reconstructed in each iteration, which makes it possible for them to take a large step. 

We present two randomized subspace gradient algorithms:
one is specialized for linear constrained optimization problems and the other is able to handle nonlinear constraints.
On linearly constrained problems, these algorithms work very similarly to the gradient projection method~\cite{rosen1960gradient} if random projections are not used.
The gradient projection method (GPM)~\cite{rosen1960gradient,du1990rosen} is one of the active set methods. Under linear constraints, GPM projects the gradients $\nabla f$ onto a subspace spanned by the gradients of the active constraints. The projected gradient does not change the function values of active constraints. 
Hence, the advantage of GPM is that all of the updated points are feasible. 
Rosen~\cite{rosen1961gradient} augmented GPM to make it able to handle nonlinear constraints and proved global convergence. 
Also for nonlinear constraints, a generalized gradient projection method (GGPM)~\cite{gao1996generalized,wang2013generalized} was derived from GPM; it does not require an initial feasible point. 

Our algorithms need $O(\frac{n}{d}\varepsilon^{-2})$ iterations to reach an approximate Karush-Kuhn-Tucker~(KKT), while
other standard gradient-algorithms for solving unconstrained problems need $O(\varepsilon^{-2})$ iterations.
While the theoretical worst-case iteration complexity increases because of the subspace projection, the gradient computation at each iteration costs less;
when gradients cannot be obtained easily, our algorithms have the same advantage as stochastic subspace gradients~\cite{kozak2021stochastic} by calculating the gradients with $d$ directional derivatives.
Due to their advantages of taking longer step and reducing the computational complexity of calculating the gradients, our algorithms are especially efficient in view of time complexity. 
This is because when gradients are obtained by directional derivatives, ours need $O(d) \times O(\frac{n}{d}\varepsilon^{-2})$ function evaluations to reach an approximate KKT, which is the same as the number of evaluations in standard gradient-algorithms for solving unconstrained problems ($O(n) \times O(\varepsilon^{-2})$).
Furthermore, our algorithms can reduce the worst time complexity to reach an approximate KKT from $O(n) \times O(\max{(\varepsilon^{-2},(u_g^{\varepsilon})^{-1}\varepsilon^{-2})})$ to $O(d) \times O(\frac{n}{d}\varepsilon^{-2})$ compared to their deterministic version when gradients are obtained by directional derivatives. 
$u_g^{\varepsilon}$ represents a particular value that is dependent on the given constraints and the parameter $\varepsilon$. In the case where all constraints are linear functions, $u_g^{\varepsilon}$ simplifies to $\varepsilon$.

In numerical experiments, we show that our algorithms work well when the gradient of the objective functions cannot be calculated. 
The numerical results indicate that our algorithms tend to find better solutions than deterministic algorithms, because they can search for solutions randomly in a wide space without being affected by the luck involved in the initial solution selection.

\section{Preliminaries}
\subsection{Notations}
In this paper, we denote the optimum of (\ref{main problem}) by $x^*$.
The vector norm $\|\cdot\|$ is assumed to be the $l_2$ norm and the matrix norm is the operator norm. $\|\cdot\|_F$ denotes the Frobenius norm. $\lambda_{\mathrm{min}}(A)$ and $\lambda_{\mathrm{max}}(A)$ are the minimum eigenvalue and maximum eigenvalue of a matrix $A$, respectively.

For a vector $a$, we define 
$[a]_+$ be a vector whose $i$-th entry is given by $  [a_i]_+ := \max(0,a_i)$. 
$\bm{1}$ denotes an all-ones vector and $\chi_S(x)$ denotes a step function for all $x \in \mathbb{R}$;
\[
	\chi_S(x) = \left\{
		\begin{array}{cc}
			1 & (x\in S),\\
			0 & (x\notin S).\\
	\end{array}\right.
\]
$\chi_{-}$ and $\chi_{+}$ are step functions with $S = \{x \vert x\le 0\}$ and $S = \{x\vert x > 0\}$, respectively. 
For a vector $a$, we define $\chi_{+}(a)$ (or  $\chi_{-}(a)$) to be a vector whose $i$-th entry is $\chi_{+}(a_i)$ (or $\chi_{-}(a_i)$, respectively).

\subsection{Key lemma}
The following lemma implies that a random projection defined by a random matrix $P$ nearly preserves the norm of any given vector $x$ with arbitrarily high probability. 
It is a variant of the Johnson-Linderstrauss lemma \cite{jllemma}.
\begin{lemma}{\label{lemma:Johnson}}~\cite{vershynin2018high}
	Let $P\in \mathbb{R}^{d\times n}$ be a random matrix whose entries are independently drawn from $\Normal$.
	Then for any $x\in \mathbb{R}^n$ and $\varepsilon\in (0,1)$, we have 
	\[
	\mathrm{Prob}[(1-\varepsilon)\|x\|^2\le \frac{1}{d}\|Px\|^2\le (1+\varepsilon)\|x\|^2]\ge 1-2\exp{(-C_0\varepsilon^2 d)},
	\]
	whose $C_0$ is an absolute constant.
\end{lemma}

\section{Algorithm for linear inequality constraints}\label{section:Algorithm for linear inequality constraints}
In this section, we describe our randomized subspace gradient method for solving linear constrained problems (RSG-LC),
which is listed as Algorithm~\ref{Secondary Algorithm} below.

\begin{algorithm}[tb]
	\caption{Randomized Subspace Gradient Method for Linear Constrained Problems (RSG-LC)}
	\label{Secondary Algorithm}
	\begin{algorithmic}
		\State INPUT: $x_0$ (a feasible point), step size $h>0$, reduced dimension $d$, $\varepsilon_0,\delta_1,\varepsilon_2>0, 1>\beta >0$.
		\For{$k=0,1,2,...$}
		\State $\alpha_k = h$
		\State Set $M_k = \frac{1}{n}P_k^\top$ by sampling a random matrix $P_k\in \mathbb{R}^{d\times n} \sim \No$.
                \State Set $G_k = \left(\nabla g_i(x_k)\right)_{i\in \mathcal{A}_k}\in \mathbb{R}^{n\times \vert \mathcal{A}_k \vert}$ by using $\mathcal{A}_k$ defined in \eqref{eq:active_set}.
\State Compute
\begin{equation}
\lambda^{(k)} = -(G_k^\top M_kM_k^\top G_k)^{-1}G_k^\top M_kM_k^\top \fgrad
\label{def_lambda}
\end{equation}
and
\begin{equation}
  d_k = -M_k^\top \left(\fgrad+G_k\lambda^{(k)}\right).
 \label{def_dk_1}
\end{equation}
		\If {$\|d_k\|\le \delta_1$}
		\If {$\min_i \lambda^{(k)}_i\ge -\varepsilon_2$}
		\State terminate.
		\Else
		\State
                \begin{equation}
                  d_k = - \frac{d}{n}M_k^\top G_k(G_k^\top M_k M_k^\top G_k)^{-1} \left[-\lambda^{(k)}\right]_+
                  \label{def_dk_2}
                \end{equation}
		\EndIf
		\EndIf
		\While{$x_k + \alpha_kM_kd_k$ is not feasible}
		\State $\alpha_k\leftarrow \beta \alpha_k$
		\EndWhile
		\State $x_{k+1}\leftarrow x_k + \alpha_kM_kd_k$
		\EndFor
	\end{algorithmic}
\end{algorithm}

\subsection{Outline of our algorithm: RSG-LC}

Let $\No$ denote the set of Gaussian matrices of size $d\times n$ whose entries are independently sampled from $\Normal$,
and define $M_k = \frac{1}{n}P_k^\top$, where $P_k$ is a Gaussian random matrix sampled from $\No$.

\paragraph{Definition of $\mathcal{A}_k$}
Let $\mathcal{A}_k :=\mathcal{A}(x_k) $ be the index set of active constraints such that the inequality constraints are almost
satisfied by equality at the iterate $x_k$.
The active set usually refers to the set of indices whose constraints satisfy $g_i(x)=0$.
We use loose criteria for the active set of our randomized subspace algorithm:
\begin{equation}
  \mathcal{A}(x) = \left\{i \in \{1,2,\cdots, m\}  \vert \  %\vert g_i(x) \vert
[-g_i(x)]_+
\le \varepsilon_0 \| \nabla g_i(x)\| \right\}
\label{eq:active_set}
\end{equation}
using a parameter $\varepsilon_0$ $(>0)$.

Section~\ref{sec:feasibility} shows that the definition of $\mathcal{A}_k$ makes the iterates $\{x_k\}$
all feasible with high probability, and
that the step size  $\alpha_k$ is larger than when the random matrices are not used.
Notice that 
we can replace $\nabla g_i$ in $\mathcal{A}_k$ with the projected gradients $M_k^\top \nabla g_i$ to reduce the computation costs when calculating $\nabla g_i$ is difficult.
This is because, from Lemma~\ref{lemma:Johnson},
$\frac{[-g_i(x_k)]_+}{\|M_k^\top \nabla g_i(x_k)\|} \approx \frac{n}{\sqrt{d}}\frac{[-g_i(x_k)]_+}{\|\nabla g_i(x_k)\|}$
holds with high probability. 
Therefore, the active set $\mathcal{A}_k$ using projected gradients can be almost the same as the set using the full gradients by adjusting $\varepsilon_0$. 
For the sake of simplicity, we define $\frac{[-g_i(x_k)]_+}{\|\nabla g_i(x_k)\|} \leq \varepsilon_0$ to be the active set $\mathcal{A}_k$.

\paragraph{Update of Iterates}
Algorithm~\ref{Secondary Algorithm} calculates the sequence $\{x_k\}$ by using the step size $\alpha_k$ and
the Lagrange multiplier $\lambda^{(k)} \in \mathbb{R}^{\vert \mathcal{A}_k \vert}$ of \eqref{def_lambda} corresponding to
  constraints in $\mathcal{A}_k$ as 
\begin{eqnarray}\label{eq:x{k+1}-x_k}
		\notag
		&x_{k+1}&= x_{k}- \alpha_kM_kM_k^\top \left(\fgrad + G_k\lambda^{(k)}\right)\\
		&&= x_{k} -\alpha_kM_k\left( I - M_k^\top G_k(G_k^\top M_kM_k^\top G_k)^{-1}G_k^\top M_k\right)M_k^\top \fgrad
\end{eqnarray}
or 
\begin{eqnarray}\label{eq:x{k+1}-x_k_2}
x_{k+1} = x_k - \frac{d}{n} \alpha_kM_k M_k^\top G_k(G_k^\top M_k M_k^\top G_k)^{-1} \left[-\lambda^{(k)}\right]_+,
\end{eqnarray}
where $G_k = \left(\nabla g_i(x_k)\right)_{i\in \mathcal{A}_k}\in \mathbb{R}^{n\times \vert \mathcal{A}_k \vert}$.
This corresponds to taking $d_k$ as \eqref{def_dk_1} or \eqref{def_dk_2}.

The search directions $x_{k+1}- x_k$ in \eqref{eq:x{k+1}-x_k} and \eqref{eq:x{k+1}-x_k_2} are descent
directions with high probability as shown in
Propositions~\ref{Proposition:dec_obj} and \ref{Proposition:obj_dec_2} in Section~\ref{sec:decreas_obj}.
Propositions~\ref{proposition:Feasibility with linear inequality} and \ref{proposition:Feasibility with linear inequality2} in Section~\ref{sec:feasibility} ensure that if the $\alpha_k$ are chosen to be less than some threshold, all points in $\{x_{k}\}$ are feasible. Hence, the while-loop reducing the value of $\alpha_k$ stops after a fixed number of iterations.
Theoretically, \eqref{eq:x{k+1}-x_k} reduces the objective function more than \eqref{eq:x{k+1}-x_k_2} does
at each iteration.  
Therefore, the algorithm first computes \eqref{eq:x{k+1}-x_k} and if the search direction is small enough and
the Lagrange multiplier $\lambda^{(k)}$ is not nonnegative, \eqref{eq:x{k+1}-x_k_2} is used for the next iterate $x_{k+1}$. 

The search direction in \eqref{eq:x{k+1}-x_k} is regarded as the projected gradient vector 
onto a subspace that itself is a random projection of the subspace spanned by the gradients of active constraints.
Note that if we ignore the random matrix $M_k$ (i.e., by setting $M_k=I$), the update rule of \eqref{eq:x{k+1}-x_k} becomes
\[
x_{k+1} = x_k - \alpha_k  \left(I- G_k  (G_k^\top G_k)^{-1}G_k^\top \right)\fgrad
\]
and $\left(I- G_k  (G_k^\top G_k)^{-1}G_k^\top \right)\fgrad$ is 
the projected gradient onto the subspace spanned by the gradients of the active constraints in $\mathcal{A}_k$.
If $\mathcal{A}_k$ is an active set in the usual sense
defined by the valid constraints, the projected gradient is the same as the one used in the gradient projection method (GPM)~\cite{rosen1960gradient}. 

By the definition of $d_k$ of (\ref{def_dk_1}), if $\|d_k\| \approx 0$, then $\|\fgrad + G_k \lambda^{(k)}\|\approx 0$ holds by Lemma~\ref{lemma:Johnson}. 
If $\lambda^{(k)} \ge 0$ holds, the resulting point is an approximate Karush-Kuhn-Tucker (KKT) point.
As shown later in Theorem~\ref{theorem:global convergence},
given inputs $d, \varepsilon_0,\varepsilon_2$ and $\delta_1 = \sqrt{\frac{d(1-\varepsilon)}{n^2}}\varepsilon_1$,
Algorithm~\ref{Secondary Algorithm} generates
for the original problem \eqref{main problem} with large $n$ an output $x_k$ that is with high probability an $(\varepsilon_1,\varepsilon_2,O(\varepsilon_0))$-KKT point
within $O(\frac{n}{d \min(\varepsilon_1^2,\varepsilon_2^2)})$ iterations.

\subsection{Decrease in objective value}\label{sec:decreas_obj}
We will make the following assumptions.
\begin{assume} 
  \begin{enumerate}[label=(\roman*)]
		\item $f$ is $L$-smooth, i.e., $\|\nabla f(x)-\nabla f(y)\| \leq L\|x-y\|$ for any $x,y \in \mathbb{R}^n$.\label{Objective is L-smooth}
		\item The level set $\{x\in \mathcal{C}\vert f(x)\le f(x_0)\}$ is compact. \label{Level Set is compact or feasible set is compact}
	\end{enumerate}
\label{Assumption:Objective}
\end{assume}
\ref{Objective is L-smooth} and \ref{Level Set is compact or feasible set is compact} together imply that
$\|\fgrad\|$ is bounded, i.e., $\|\fgrad\|\le U_f$ for some value $U_f>0$.

\begin{assume}\label{Assumption:existence of inverse}
  \begin{enumerate}[label=(\roman*)]
    \item
  The vectors $\nabla g_i(x)$, $\forall i \in \mathcal{A}_k$,  of the active constraints at $x_k \in \mathcal{C}$ are linearly independent. \label{Constraints are independent}
\item The reduced dimension $d$ is larger than the number of active constraints $\vert \mathcal{A}_k\vert$.
  \label{Assumption:reduced dimension is larger than active sets}
  \end{enumerate}
  \label{Assumption:MatRank}
\end{assume}

Notice that Assumption~\ref{Assumption:MatRank} implies that $M_k^\top G_k$ has full row rank with probability 1,
which ensures the existence of $(G_k^\top M_kM_k^\top G_k)^{-1}$.
Here,  we define 
\[ 
\lambda^*_{\mathrm{min}} =
\min_{\substack{\mathcal{A} \subset \{1,2,...,m\},\vert \mathcal{A} \vert <n, x\in \mathcal{C}, \\
\nabla g_i(x) (i\in \mathcal{A}) \mathrm{\;are\;linear\;independent}}}\lambda_{\mathrm{min}}(G_{\mathcal{A}}(x)^\top G_{\mathcal{A}}(x)),
\]
where $G_\mathcal{A}(x)  = (\nabla g_i(x))_{i \in \mathcal{A}}$.
Obviously, $\lambda^*_{\mathrm{min}} > 0$ and
Assumption~\ref{Assumption:MatRank}\ref{Constraints are independent} ensure that $\lambda^*_{\mathrm{min}} \leq \lambda_{\mathrm{min}}(G_k^\top G_k)$.
By defining 
\begin{equation}
	\label{def:Z_k}
	Z_k:= I-M_k^\top G_k (G_k^\top M_kM_k^\top G_k)^{-1} G_k^\top M_k,
\end{equation}
we can rewrite the next iterate (\ref{eq:x{k+1}-x_k}) as $x_{k+1} = x_k - \alpha_k M_k Z_k M_k^\top \fgrad$.

\begin{proposition}\label{Proposition:dec_obj} 
  Under Assumptions~\ref{Assumption:Objective}\ref{Objective is L-smooth} and \ref{Assumption:MatRank},  
	when $d_k = - M_k^\top \left(\fgrad+G_k\lambda^{(k)}\right)$ of \eqref{def_dk_1} in Algorithm~\ref{Secondary Algorithm}, then
	\[
		f(x_{k+1})\le f(x_k)-\left(\alpha_k - \frac{L\alpha_k^2(1+\varepsilon)}{2n}\right)\|Z_kM_k^\top \fgrad\|^2
	\]
	holds with probability at least $1-2\exp{(-C_0 \varepsilon^2 n)}$. 
\end{proposition}
\begin{proof}
	Since $f$ is $L$-smooth from Assumption~\ref{Assumption:Objective}\ref{Objective is L-smooth}, we have
	\begin{equation}\label{eq:L-smooth}
		f(x_{k+1})-f(x_k) \le \langle \fgrad, x_{k+1}-x_{k} \rangle  +\frac{L}{2} \|x_{k+1}-x_{k}\|^2.
	\end{equation}
	We also have $x_{k+1}-x_{k}=- \alpha_k M_k Z_k M_k^\top \fgrad$ from (\ref{eq:x{k+1}-x_k}).
	Using this equality in (\ref{eq:L-smooth}), we obtain
	\begin{align*}
		\notag
		&f(x_{k+1})-f(x_k) \le  \\
		&\quad -\alpha_k \fgrad^\top M_kZ_kM_k^\top \fgrad +\frac{\alpha_k^2L}{2}\|M_kZ_kM_k^\top \fgrad\|^2.
	\end{align*}
	By definition, $Z_k$ is an orthogonal projection matrix; hence, $Z_k^2 = Z_k$. Therefore, $\fgrad^\top M_kZ_kM_k^\top \fgrad=\|Z_kM_k^\top \fgrad\|^2$.
	Furthermore, by noticing that $ M_k = \frac{1}{n}P_k^\top$ and using Lemma~\ref{lemma:Johnson},
	we obtain that
	$ \| M_kZ_kM_k^\top \fgrad\|^2 = \frac{1}{n^2}\|P_k^\top Z_kM_k^\top \fgrad \|^2\le  \frac{1+\varepsilon}{n}\|Z_kM_k^\top \fgrad\|^2.$
	With these relations, we find that
	\[
	f(x_{k+1})-f(x_k) \le -\left(\alpha_k -\frac{\alpha_k^2L(1+\varepsilon)}{2n} \right) \|Z_kM_k^\top \fgrad\|^2
	\]
	with probability at least $1-2\exp{(-C_0\varepsilon^2n)}$.
\end{proof}

Because $\|Z_kM_k^\top \fgrad\|=\|M_k^\top (\fgrad + G_k\lambda^{(k)})\|$, which is derived from \eqref{def_lambda} and \eqref{def:Z_k},
Proposition~\ref{Proposition:dec_obj} implies that the function value $f(x_k)$ strictly decreases
unless $d_k = 0$.

\begin{proposition}\label{Proposition:obj_dec_2}
	Under Assumptions~\ref{Assumption:Objective}\ref{Objective is L-smooth} and \ref{Assumption:MatRank},
	when $d_k = - \frac{d}{n}M_k^\top G_k(G_k^\top M_k M_k^\top G_k)^{-1} [-\lambda^{(k)}]_+$
	of \eqref{def_dk_2} in Algorithm~\ref{Secondary Algorithm}, then
	\[
	f(x_{k+1})\le f(x_k) - \frac{d}{n}\left(\alpha_k - \frac{L\alpha_k^2(1+\varepsilon)}{2(1-\varepsilon)\lambda_{\mathrm{min}}(G_k^\top G_k)}\right)\|[-\lambda^{(k)}]_+\|^2
	\]
	with probability at least $1-2\exp{(-C_0\varepsilon^2n)}-2\exp{(-C_0\varepsilon^2d)}$.
\end{proposition}
\begin{proof}
We can confirm that $M_k d_k$ is a descent direction by using the definition \eqref{def_lambda} of $\lambda^{(k)}$ for the second equality:
\begin{align}{\label{eq:descent direction2 in Algorithm2}}
	\notag
	\fgrad^\top M_k d_k 
	&= -\frac{d}{n}\fgrad^\top M_k M_k^\top G_k(G_k^\top M_k M_k^\top G_k)^{-1} [-\lambda^{(k)}]_+\\
	& \notag =\frac{d}{n}\lambda^{(k)\top} [-\lambda^{(k)}]_+ \\
	&= -\frac{d}{n}\|[-\lambda^{(k)}]_+\|^2 <0.
\end{align}

Now let us evaluate $f(x_{k+1}) - f(x_k)$ using \eqref{eq:L-smooth} 
with $x_{k+1}-x_k=\alpha_k M_kd_k$ as %$L$-smoothness assumption 
\begin{equation}		
		f(x_{k+1})  \le f(x_k) + \alpha_k\fgrad^\top M_k d_k +\frac{L\alpha_k^2}{2}\| M_k d_k\|^2. 
\label{eq:funcdic_Lsmooth}
\end{equation}
Notice first that $M_k^\top M_k = \frac{1}{n^2}P_kP_k^\top$ and Lemma~\ref{lemma:Johnson} give
	\begin{equation}{\label{eq:M_k d_k}}
		\| M_k d_k \|^2 \le \frac{1+\varepsilon}{n} \|d_k\|^2
	\end{equation}
	with probability at least $1-2\exp{(-C_0\varepsilon^2n)}$. Moreover, 
	\begin{align}{\label{eq:norm_dk}}
		\|d_k\|^2 &= \frac{d^2}{n^2} ([-\lambda^{(k)}]_+)^\top (G_k^\top M_k M_k^\top G_k)^{-1} [-\lambda^{(k)}]_+ \notag \\
	 &\le \frac{d^2}{n^2\lambda_{\min}(G_k^\top M_k M_k^\top G_k)}  \|[-\lambda^{(k)}]_+\|^2
	\end{align}
	holds from the definition of $d_k$. Let $w$ be the eigenvector corresponding to the minimum eigenvalue of $G_k^\top M_k M_k^\top G_k$. From Lemma~\ref{lemma:Johnson}, 
we have 
	\begin{align}{\label{eq:eig_linear}}
		\notag
		\lambda_{\mathrm{min}}(G_k^\top M_k M_k^\top G_k) &= w^\top (G_k^\top M_k M_k^\top G_k) w \\
		& \notag \ge \frac{d(1-\varepsilon)}{n^2} w^\top (G_k^\top G_k) w \\
		&\ge  \frac{d(1-\varepsilon)}{n^2} \lambda_{\mathrm{min}}(G_k^\top G_k)
	\end{align}
	with probability at least $1-2\exp{(-C_0\varepsilon^2d)}$. 
	Combining (\ref{eq:norm_dk}) and (\ref{eq:eig_linear}), we find that
	\begin{equation}{\label{eq:norm of second direction}}
		\|d_k\|^2 \le \frac{d}{(1-\varepsilon)\lambda_{\mathrm{min}}(G_k^\top G_k)}\|[-\lambda^{(k)}]_+\|^2. 
	\end{equation}
	From Assumption~~\ref{Assumption:MatRank}\ref{Constraints are independent}, $\lambda_{\min}(G_k^\top G_k)$ has a positive lower bound.
	\eqref{eq:funcdic_Lsmooth} together with (\ref{eq:descent direction2 in Algorithm2}) leads to
	\begin{alignat*}{2}
		\notag
		f(x_{k+1}) & \le 
 		f(x_k) - \alpha_k \frac{d}{n}\|[-\lambda^{(k)}]_+\|^2+\frac{L\alpha_k^2}{2}\| M_k d_k\|^2 &\\
		\notag
		& \le f(x_k) - \alpha_k \frac{d}{n}\|[-\lambda^{(k)}]_+\|^2 + \frac{L\alpha_k^2(1+\varepsilon)}{2n}\|d_k\|^2 &\\
		&\le  f(x_k) - \alpha_k \frac{d}{n}\|[-\lambda^{(k)}]_+\|^2 + \frac{dL\alpha_k^2(1+\varepsilon)}{2n(1-\varepsilon)\lambda_{\mathrm{min}}(G_k^\top G_k)}\|[-\lambda^{(k)}]_+\|^2,&
	\end{alignat*}
	with probability at least $1-2\exp{(-C_0\varepsilon^2n)}-2\exp{(-C_0\varepsilon^2d)}$. 
The second inequality follows from (\ref{eq:M_k d_k}) and the last inequality follows from (\ref{eq:norm of second direction}).
\end{proof}
Proposition~\ref{Proposition:obj_dec_2} shows that if $0 < \alpha_k <\frac{2(1-\varepsilon)\lambda_{\mathrm{min}}(G_k^\top G_k)}{(1+\varepsilon)L}$ and $\min_{i}\lambda_i^{(k)} < 0$, then $f(x_{k+1}) < f(x_k)$.

\subsection{Feasibility}\label{sec:feasibility}
In this section, we derive conditions under which that the sequence $\{x_k\}$ generated by Algorithm~\ref{Secondary Algorithm} is feasible. 
Assuming that the initial point is feasible, we update the iterate while preserving feasibility.
We prove the following lemma by utilizing the properties of linear constraints.
\begin{proposition}{\label{proposition:Feasibility with linear inequality}}
  Suppose that Assumptions~\ref{Assumption:Objective} and \ref{Assumption:MatRank} hold and  
	that all constraints $g_i$ are linear
  and $x_k$ is feasible.
	Then, if the step size $\alpha_k$ satisfies
	\[
	0\le \alpha_k \le \frac{\varepsilon_0}{U_f}\frac{n^2}{(1+\varepsilon)d}
	\]
	when $d_k = - M_k^\top (\nabla f(x_k)+ G_k \lambda^{(k)})$ of \eqref{def_dk_1} in Algorithm~\ref{Secondary Algorithm},
	$x_{k+1}$ is feasible with probability at least $1-2(m-\vert\mathcal{A}_k\vert+1)\exp{(-C_0\varepsilon_1^2d)}$.
\end{proposition}
\begin{proof}
Note that  $\lambda^{(k)}$ is the solution of 
$G_k^\top M_kM_k^\top G_k \lambda^{(k)} = -G_k^\top M_kM_k^\top \fgrad$,
which is equal to $G_k^\top M_k d_k =0$ when  $d_k$ is defined as \eqref{def_dk_1}.
Since the constraints are linear, the direction $M_kd_k$ preserves feasibility for the active constraints: 
\begin{equation*}
g_i(x_k +\alpha_k M_kd_k) = g_i(x_k) +\alpha_k \nabla g_i(x_k)^\top  M_k d_k = g_i(x_k), \;\;
\forall i \in \mathcal{A}_k.
\end{equation*}
  As for $i\notin \mathcal{A}_k$, notice that $x_{k+1}$ is feasible if 
	\begin{align*}
 g_i(x_{k+1}) &= g_i(x_k) - \alpha_k\nabla g_i(x_k)^\top  M_k Z_k M_k^\top \fgrad \\
		&\le g_i(x_k) + \alpha_k\|M_k^\top \nabla g_i(x_k)\| \|Z_k M_k^\top \fgrad\|
		\le 0
	\end{align*}
	is satisfied.
	By Lemma~\ref{lemma:Johnson}, we have
	\[
	\|M_k^\top \nabla g_i(x_k)\| \le \sqrt{\frac{d(1+\varepsilon)}{n^2} }\|\nabla g_i(x_k)\| \qquad(i\notin \mathcal{A}_k),
	\]
	with probability at least $1 - 2(m-\vert\mathcal{A}_k\vert) \exp{(-C_0\varepsilon^2d)}$.
	Since $Z_k$ is an orthogonal projection, we further have
	\[
	\|Z_kM_k^\top \nabla f(x_k)\| \le\|M_k^\top \nabla f(x_k)\|\le \sqrt{\frac{d(1+\varepsilon)}{n^2}} \|\nabla f(x_k)\| 
	\]
	with probability at least $1-2\exp{(-C_0\varepsilon^2d)}$. From Assumption~\ref{Assumption:Objective}\ref{Level Set is compact or feasible set is compact}, we have
	\begin{flalign} 
             \notag     g_i(x_k) + \alpha_k\|M_k^\top \nabla g_i(x_k)\| \|Z_k M_k^\top \fgrad\| 
		&\notag \le g_i(x_k) + \alpha_k\frac{d(1+\varepsilon)}{n^2}\|\nabla g_i(x_k)\|\|\fgrad\| & \\
		&\le g_i(x_k) + \alpha_k \frac{d(1+\varepsilon)}{n^2}\|\nabla g_i(x_k)\| U_f.
\label{eq:feas_check}
	\end{flalign}
    Because of the assumption on $\alpha_k$ and $\varepsilon_0 < \frac{-g_i(x_k)}{\|\nabla g_i(x_k)\|}$ for $i\notin \mathcal{A}_k$, we have
	\[
	0\le \alpha_k \le \frac{\varepsilon_0}{U_f}\frac{n^2}{(1+\varepsilon)d}
< \frac{1}{U_f}\frac{n^2}{(1+\varepsilon)d} \frac{-g_i(x_k)}{\|\nabla g_i(x_k)\|},
	\]
which ensures that the right-hand side of \eqref{eq:feas_check} is upper-bounded by $0$. Hence, $x_{k+1}$ is feasible.
\end{proof}
\begin{remark}
	If we do not use random matrices, the condition becomes $\alpha_k \le \varepsilon_0/U_f$ and the original dimension $n$ does not appear. Also, since $\|M_k^\top (\nabla f(x_k)+ G_k \lambda^{(k)})\|\approx \frac{\sqrt{d}}{n}\|\nabla f(x_k)+ G_k \lambda^{(k)}\|$, we see that, using a random subspace allows us to use a larger step size $\frac{n}{\sqrt{d}}$.
	Notice, that when the dimension $n$ is large enough,
	the random matrices allow us to ignore the step-size condition because the step size is at least proportional to $\sqrt{n} ( < \frac{n}{\sqrt{d}} )$, which comes from $d < n$.
\end{remark}

\begin{proposition}{\label{proposition:Feasibility with linear inequality2}}
	Suppose that Assumptions~\ref{Assumption:Objective} and \ref{Assumption:MatRank} hold and that all constraints $g_i$ are linear
  and $x_k$ is feasible.
	Then, if the step size $\alpha_k$ satisfies
	\begin{equation}{\label{eq:lower bound of step size with 2nd direction}}
		0\le \alpha_k \le \frac{n}{d}\frac{(1-\varepsilon)\lambda_{\mathrm{min}}^*}{(1+\varepsilon)U_f}\varepsilon_0 
	\end{equation}
	when $d_k = - \frac{d}{n}M_k^\top G_k(G_k^\top M_k M_k^\top G_k)^{-1} [-\lambda^{(k)}]_+$
of \eqref{def_dk_2} in Algorithm~\ref{Secondary Algorithm}, %then the point 
$x_{k+1}$ is feasible with probability at least $1-2(m-\vert\mathcal{A}_k\vert+2)\exp{(-C_0\varepsilon^2d)}$.
\end{proposition}
\begin{proof}
	Regarding the active constraints, 
	\begin{align*}
		G_k^\top M_kd_k 
		=& - \frac{d}{n}G_k^\top M_kM_k^\top G_k(G_k^\top M_k M_k^\top G_k)^{-1} [-\lambda^{(k)}]_+\\
		=& -\frac{d}{n}[-\lambda^{(k)}]_+ \le 0
	\end{align*}
	
	holds. Therefore,
	\[
		g_i(x_k +\alpha_k M_kd_k) = g_i(x_k) +\alpha_k \nabla g_i(x_k)^\top M_k d_k \leq  g_i(x_k), \;\;
		\forall i \in \mathcal{A}_k.
	\]
	Hence, $g_i(x_{k+1})\le 0$ is satisfied for the active constraints. 

As for the nonactive constraints ($i\notin \mathcal{A}_k$), it is enough to prove the inequality,
\begin{equation}
	 g_i(x_k) + \alpha_k \|M_k^\top\nabla g_i (x_k)\|\|d_k\|\le 0,
\label{eq:feas_alpha_dir2}
\end{equation}
which leads to
	\[
	g_i(x_{k+1}) = g_i(x_k) + \alpha_k\nabla g_i^\top (x_k) M_kd_k 
	\le g_i(x_k) + \alpha_k \|M_k^\top\nabla g_i (x_k)\|\|d_k\|\le 0.
	\]
	From Lemma~\ref{lemma:Johnson}, we find that 
	\begin{alignat}{2}
		& \notag g_i(x_k) + \alpha_k \|M_k^\top\nabla g_i (x_k)\|\|d_k\| &\\
		&  \notag \le g_i(x_k) + \alpha_k \sqrt{\frac{d(1+\varepsilon)}{n^2}}\|\nabla g_i (x_k)\|\|d_k\| \\
		&  \notag \le g_i(x_k) + \alpha_k \frac{d}{n}\sqrt{\frac{(1+\varepsilon)}{(1-\varepsilon)\lambda_{\mathrm{min}}(G_k^\top G_k)}} \|\nabla g_i(x_k)\|\|[-\lambda^{(k)}]_+\| \\
		&\le g_i(x_k) + \alpha_k \frac{d}{n}\sqrt{\frac{(1+\varepsilon)}{(1-\varepsilon)\lambda_{\mathrm{min}}^*}} \|\nabla g_i(x_k)\|\|[-\lambda^{(k)}]_+\|
		\qquad(i\notin \mathcal{A}_k) 
\label{eq:nonact_const}
	\end{alignat}
	holds with probability at least $1-2(m-\vert\mathcal{A}_k\vert+1)\exp{(-C_0\varepsilon_1^2d)}$. 
The second inequality follows from (\ref{eq:norm of second direction}).
Now, we will show that the step size satisfies
	\begin{alignat}{2}
		0\le
	\alpha_k & \notag \le \frac{n}{d}\frac{(1-\varepsilon)\lambda_{\mathrm{min}}^*}{(1+\varepsilon)U_f}\varepsilon_0 \\
 & \le \frac{n}{d\|[-\lambda^{(k)}]_+\|} \sqrt{\frac{(1-\varepsilon)\lambda_{\mathrm{min}}^*}{(1+\varepsilon)}}\varepsilon_0,
\label{eq:rel_alpha}
	\end{alignat} 
which proves \eqref{eq:feas_alpha_dir2} from $\varepsilon_0 \le \frac{-g_i(x_k)}{\|\nabla g_i (x_k)\|}$ together with \eqref{eq:nonact_const}.
Let us evaluate the upper bound on $\|\lambda^{(k)}\|$. From \eqref{def_lambda},  
we have
	\begin{align*}		
		\|\lambda^{(k)}\|& = \|(G_k^\top M_k M_k^\top G_k)^{-1}G_k^\top M_k M_k^\top \fgrad \|\\
		&\le  \|(G_k^\top M_k M_k^\top G_k)^{-1}G_k^\top M_k\|\| M_k^\top \fgrad \|,               
	\end{align*}
where the first norm  $\|\cdot\|$ on the right-hand side
         is the operator norm for a matrix. We obtain  
	\[
	\|(G_k^\top M_k M_k^\top G_k)^{-1}G_k^\top M_k\|^2
	= \lambda_{\mathrm{max}}((G_k^\top M_k M_k^\top G_k)^{-1})
	= \frac{1}{\lambda_{\mathrm{min}}(G_k^\top M_k M_k^\top G_k)}.
	\]
	Then, from (\ref{eq:eig_linear}) and Lemma~\ref{lemma:Johnson}, we find that 
	\begin{alignat}{2}{\label{eq:upper bound of lambda}}
		\notag
\|\lambda^{(k)}\|
		&\le  \frac{\|M_k^\top \fgrad\|}{\sqrt{\lambda_{\mathrm{min}}(G_k^\top M_k M_k^\top G_k)}} &\\
		\notag 
		&\le \sqrt{\frac{n^2}{d(1-\varepsilon)}}\frac{\|M_k^\top \fgrad\|}{\sqrt{\lambda_{\mathrm{min}}(G_k^\top G_k)}} \\
		\notag
		&\le \sqrt{\frac{(1+\varepsilon)}{(1-\varepsilon)}} \frac{\|\fgrad\|}{\sqrt{\lambda_{\mathrm{min}}(G_k^\top G_k)}} \\ 
		&\le\sqrt{\frac{(1+\varepsilon)}{(1-\varepsilon)}} \frac{U_f}{\sqrt{\lambda_{\mathrm{min}}(G_k^\top G_k)}}
		\le \sqrt{\frac{(1+\varepsilon)}{(1-\varepsilon)}} \frac{U_f}{\sqrt{\lambda_{\mathrm{min}}^*}} &
	\end{alignat}
	holds with probability at least $1-4\exp{(-C_0\varepsilon^2d)}$. 
	The second inequality follows from (\ref{eq:eig_linear}) and the third inequality follows from Lemma~\ref{lemma:Johnson}.
Hence, 	using  \eqref{eq:upper bound of lambda} with $\|[-\lambda^{(k)}]_+\| \leq \|\lambda^{(k)}\|$,
the second inequality in \eqref{eq:rel_alpha} holds and it is confirmed that $x_{k+1}$ satisfies the nonactive inequality constraints.
Thus, $x_{k+1}$ is feasible 
with probability at least $1-2(m-\vert\mathcal{A}_k\vert+2)\exp{(-C_0\varepsilon^2d)}$;
the probability can be derived by applying Lemma~\ref{lemma:Johnson} to $\|M_k^\top \fgrad\|, \|M_k^\top \nabla g_i(x_k)\|\;(i \notin \mathcal{A}_k)$ and (\ref{eq:eig_linear}).
\end{proof}

	Similarly to Proposition~\ref{proposition:Feasibility with linear inequality}, we can ignore this condition when the original dimension $n$ is large enough.
\subsection{Global convergence}\label{sec:global_conv}
\begin{definition}
	We say that $(x,\eta) \in \mathbb{R}^n \times \mathbb{R}^m$ is an $(\varepsilon_1,\varepsilon_2,\varepsilon_3)$-KKT pair of
 Problem (\ref{main problem}) if the following conditions hold:
	\begin{eqnarray}
		\label{definite:epsilon kkt 1} \|\nabla f(x) + \sum_{i = 1}^{m}\eta_i \nabla g_i(x)\|\le \varepsilon_1,\\
		\label{definite:feasible}g_i(x)\le 0,\\
		\label{definite:epsilon kkt 2}\eta_i \ge -\varepsilon_2,\\
		\label{definite:epsilon kkt 3} \vert \eta_i g_i(x) \vert \le \varepsilon_3.
	\end{eqnarray}
\end{definition}
We can construct $\eta^{(k)}  \in \mathbb{R}^m$ from the output $\lambda^{(k)}  \in \mathbb{R}^{\vert \mathcal{A}_k \vert}$ 
of Algorithm~\ref{Secondary Algorithm} as follows:
 copy the values of $\lambda^{(k)}$  to the subvector of $\eta^{(k)}$ corresponding to the index set $\mathcal{A}_k$, filling in
the other elements of $\eta^{(k)}$ with $0$. We will regard $(x_k, \eta^{(k)})$ as the output of  Algorithm~\ref{Secondary Algorithm}. 
Now, let us prove that the output $(x_k,\eta^{(k)})$ is an $(\varepsilon_1,\varepsilon_2,\varepsilon_3)$-KKT pair for some 
$\varepsilon_1,\varepsilon_2,\varepsilon_3$.

\begin{theorem}{\label{theorem:global convergence}}
Suppose that Assumptions~\ref{Assumption:Objective} and \ref{Assumption:MatRank} hold.
	Let the constraints $g_i$ be linear functions and $U_g = \max_i \|\nabla g_i(x)\|$. Moreover, let the optimal value of (\ref{main problem}) be $f^*(> -\infty)$, and
	let 
	\begin{eqnarray*}
		&\delta(\varepsilon,\varepsilon_0,\varepsilon_1,\varepsilon_2) & := \min\left(\min\left(\frac{n}{L(1+\varepsilon)}, \frac{\varepsilon_0}{U_f}\frac{n^2}{(1+\varepsilon)d}\right)\frac{d(1-\varepsilon)}{2n^2}\varepsilon_1^2,\right.\\
		&&  ~~~~~~~~~~~~~~~~~~~~ \left.\min\left(\frac{(1-\varepsilon)\lambda_{\mathrm{min}}^*}{(1+\varepsilon)L}, \frac{n}{d}\frac{(1-\varepsilon)\lambda_{\mathrm{min}}^*}{(1+\varepsilon)U_f}\varepsilon_0\right)\frac{d}{2n}\varepsilon_2^2\right).
	\end{eqnarray*}
Then  Algorithm~\ref{Secondary Algorithm} generates an $(\varepsilon_1,\varepsilon_2,O(\varepsilon_0))$-KKT pair from the inputs $d,\varepsilon_0,\varepsilon_2$, $\delta_1 = \sqrt{\frac{d(1-\varepsilon)}{n^2}}\varepsilon_1$ 
within $K := \left\lceil \frac{f(x_0)-f^*}{\delta(\varepsilon,\varepsilon_0,\varepsilon_1,\varepsilon_2)}\right\rceil $
 iterations, with
	probability at least $1 - 2K\exp{(-C_0\varepsilon^2n)}-2K(m+5)\exp{(-C_0\varepsilon^2d)}$.
\end{theorem}
\begin{proof}
	The points $\{x_{k}\}$ are feasible when the step size $\alpha_k$ satisfies the conditions of Propositions~\ref{proposition:Feasibility with linear inequality} and \ref{proposition:Feasibility with linear inequality2}. Hence, (\ref{definite:feasible}) is satisfied. Next, we prove (\ref{definite:epsilon kkt 3}).
	In terms of $i\in \mathcal{A}_k$, $\frac{[-g_i(x)]_+}{\|\nabla g_i(x)\|} \le \varepsilon_0$ and Lemma~\ref{lemma:Johnson} imply that
	\begin{alignat*}{2}
		\sum_{i\in \mathcal{A}_k} (\eta_i^{(k)})^2([-g_i(x_k)]_+)^2 &\le \sum_{i\in \mathcal{A}_k} \varepsilon_0^2 (\eta_i^{(k)})^2\|\nabla g_i(x_k)\|^2 \\
		&\le \varepsilon_0^2\max_{j} \|\nabla g_j(x_k)\|^2\sum_{i\in \mathcal{A}_k} (\eta_i^{(k)})^2 &\\
		&=\varepsilon_0^2\max_{j} \|\nabla g_j(x_k)\|^2 \|\lambda^{(k)}\|^2 &\\
		&\le \varepsilon_0^2\max_{j} \frac{\|\nabla g_j(x_k)\|^2}{\lambda_{\mathrm{min}}(G_k^\top G_k)}\lambda^{(k)\top} G_k^\top G_k \lambda^{(k) } &\\
		&\le \varepsilon_0^2 \frac{n^2}{d(1-\varepsilon)} \max_{j} \frac{\|\nabla g_j(x_k)\|^2}{\lambda_{\mathrm{min}}(G_k^\top G_k)} \lambda^{(k)\top} G_k^\top M_kM_k^\top G_k \lambda^{(k)}.&
	\end{alignat*}
	The first inequality follows from $[- g_i(x_k)]_+ \le \varepsilon_0 \|\nabla g_i(x_k)\|$ and the last inequality follows from Lemma~\ref{lemma:Johnson}.
	From \eqref{def_lambda}  
we have
	\begin{align*}
			\lambda^{(k)\top} G_k^\top M_kM_k^\top G_k \lambda^{(k)}
		&= \fgrad^\top M_k M_k^\top G_k(G_k^\top M_kM_k^\top G_k)^{-1}G_k^\top M_k M_k^\top \fgrad \\
		&\le \|M_k^\top \fgrad\|^2,
	\end{align*}
	where we have used the fact that $ M_k^\top G_k(G_k^\top M_kM_k^\top G_k)^{-1}G_k^\top M_k$ is an orthogonal projection matrix and hence its maximum eigenvalue is equal to $1$.
	From these inequalities, we obtain  
	\begin{alignat}{2}{\label{eq:eval_cond3}}
		\notag
		\sum_{i\in \mathcal{A}_k} (\eta_i^{(k)})^2([-g_i(x_k)]_+)^2 & \le \varepsilon_0^2 \frac{n^2}{d(1-\varepsilon)} \max_{j} \frac{\|\nabla g_j(x_k)\|^2}{\lambda_{\mathrm{min}}(G_k^\top G_k)} \|M_k^\top \fgrad\|^2\\
		\notag
		&\le \varepsilon_0^2 \frac{(1 + \varepsilon)}{(1-\varepsilon)} \max_{j} \frac{\|\nabla g_j(x_k)\|^2\|\fgrad\|^2}{\lambda_{\mathrm{min}}(G_k^\top G_k)} \\
		& \le \varepsilon_0^2 \frac{(1+\varepsilon)}{(1-\varepsilon)}\frac{U_g^2 U_f^2}{\lambda_{\mathrm{min}}^*}.&
	\end{alignat}
	The second inequality follows from Lemma~\ref{lemma:Johnson}.
	Hence, $\vert \eta_i^{(k)}g_i(x_k) \vert=\vert \eta_i^{(k)}[-g_i(x_k)]_+ \vert = O(\varepsilon_0)$, satisfying \eqref{definite:epsilon kkt 3}.
	Next, we derive \eqref{definite:epsilon kkt 1}. 
	When Algorithm~\ref{Secondary Algorithm} terminates at iteration $\bar{k}$, 
	 $\|M_{\bar{k}}^\top \nabla f(x_{\bar{k}}) + M_{\bar{k}}^\top G_{\bar{k}} \lambda^{(\bar{k})}\|\le \delta_1$ holds, implying from Lemma~\ref{lemma:Johnson} that 
	\[
	\sqrt{\frac{d(1-\varepsilon)}{n^2}}\| \nabla f(x_{\bar{k}}) + G_{\bar{k}} \lambda^{({\bar{k}})}\| \le \delta_1 = \sqrt{\frac{d(1-\varepsilon)}{n^2}}\varepsilon_1
	\]
	with probability at least $1-2\exp{(-C_0\varepsilon^2d)}$. By definition of $\eta^{({\bar{k}})}$, we have 
	$$\| \nabla f(x_{\bar{k}}) + G_{\bar{k}} \lambda^{({\bar{k}})}\|=\|\nabla f(x) + \sum_{i = 1}^m\eta_i \nabla g_i(x)\|.$$ Furthermore, because $\min_i \lambda_i^{(\bar{k})} > -\varepsilon_2$ also holds,  $(x_{\bar{k}},\eta^{({\bar{k}})})$ is an $(\varepsilon_1,\varepsilon_2,O(\varepsilon_0))$-KKT pair.

	Now we prove that the iteration number $\bar{k}$ for finding an $(\varepsilon_1,\varepsilon_2,O(\varepsilon_0))$-KKT pair is at most $K$, i.e., $\bar{k} \leq K$.
	We will show that the function value strictly and monotonically decreases in the two directions (denoted as Case 1 and Case 2, here)  at each iteration $k (\leq \bar{k}-1)$ of Algorithm~\ref{Secondary Algorithm}.
	\begin{description}
\item{Case 1: }
	When $\|Z_k M_k^\top \fgrad\|=\|M_k^\top \fgrad + M_k^\top G_k \lambda^{(k)}\|\ge\delta_1$ and $0 \le \alpha_k \le \frac{n}{L(1+\varepsilon)}$,
	we have 
	\[
	\alpha_k - \frac{L\alpha_k^2(1+\varepsilon)}{2n} \ge \frac{1}{2}\alpha_k.
	\]
	These relations and Proposition~\ref{Proposition:dec_obj} yield the following:
	\begin{equation*}{\label{eq:global convergence update1}}
		f(x_{k+1})\le f(x_k) -\frac{1}{2}\alpha_k\frac{d(1-\varepsilon)}{n^2}\varepsilon_1^2.
	\end{equation*}
	If the step size $\alpha_k$ satisfies the condition of Proposition~\ref{proposition:Feasibility with linear inequality},
	\begin{equation}{\label{eq:global convergence, 1}}
		f(x_{k+1}) - f(x_k) \le -\frac{1}{2}\min\left(\frac{n}{L(1+\varepsilon)}, \frac{\varepsilon_0}{U_f}\frac{n^2}{(1+\varepsilon)d}\right)\frac{d(1-\varepsilon)}{n^2}\varepsilon_1^2
	\end{equation}
	holds and $x_{k+1}$ is feasible.
\item{Case 2: }
	Next, when $\|M_k^\top \fgrad + M_k^\top G_k \lambda^{(k)}\|\le\delta_1$, we update the point by 
	\[
	d_k = -\frac{d}{n}M_k^\top G_k(G_k^\top M_k M_k^\top G_k)^{-1} [-\lambda^{(k)}]_+.
	\] 
	Since Algorithm~\ref{Secondary Algorithm} does not terminate at iteration $k$, we have $\min_i \lambda_i^{(k)}<-\varepsilon_2$, and this inequality leads to 
	$\|[-\lambda^{(k)}]_+\|^2\ge \varepsilon_2^2$. 
	When then step size $\alpha_k$ satisfies $0\le \alpha_k \le \frac{(1-\varepsilon)\lambda_{\mathrm{min}}(G_k^\top G_k)}{(1+\varepsilon)L}$, from Proposition~\ref{Proposition:obj_dec_2} and 
	\[
	\alpha_k - \frac{L\alpha_k^2(1+\varepsilon)}{2(1-\varepsilon)\lambda_{\mathrm{min}}(G_k^\top G_k)} \ge \frac12\alpha_k,
	\] 
	we have 
	\begin{equation*}{\label{eq:global convergence update2}}
		f(x_{k+1})- f(x_k) \le - \alpha_k\frac{d}{2n}\varepsilon_2^2.
	\end{equation*}
	If the step size $\alpha_k$ satisfies \eqref{eq:lower bound of step size with 2nd direction} in Proposition~\ref{proposition:Feasibility with linear inequality2}, 
	\begin{equation}{\label{eq:global convergence, 2}}
		f(x_{k+1})- f(x_k) \le - \min\left(\frac{(1-\varepsilon)\lambda_{\mathrm{min}}^*}{(1+\varepsilon)L}, \frac{n}{d}\frac{(1-\varepsilon)\lambda_{\mathrm{min}}^*}{(1+\varepsilon)U_f}\varepsilon_0\right)\frac{d}{2n}\varepsilon_2^2
	\end{equation}
	holds and $x_{k+1}$ is feasible.
\end{description}
	Combining (\ref{eq:global convergence, 1}) and (\ref{eq:global convergence, 2}) and then summing over $k$ and using the relation $f^* \leq f(x_{\bar{k}})$, we find that
	\begin{equation}
	f^* - f(x_{0}) \le f(x_{\bar{k}}) - f(x_{0}) \le -\bar{k} \delta(\varepsilon,\varepsilon_0,\varepsilon_1,\varepsilon_2)
        \label{eq:contradict_ineq}
	\end{equation}
	holds with probability as least $1 - 2\bar{k}\exp{(-C_0\varepsilon^2n)}-2\bar{k}(m+5)\exp{(-C_0\varepsilon^2d)}$.
  Accordingly, \eqref{eq:contradict_ineq} implies $\bar{k} \leq K$, which completes the proof. 
\end{proof}
\begin{remark}
If the original dimension $n$ is large enough, $\delta(\varepsilon,\varepsilon_0,\varepsilon_1,\varepsilon_2)$ becomes
	\[
		\delta(\varepsilon,\varepsilon_0,\varepsilon_1,\varepsilon_2) = \min\left(\frac{(1-\varepsilon)}{2L(1+\varepsilon)}\frac{d}{n}\varepsilon_1^2,
		\frac{(1-\varepsilon)\lambda_{\mathrm{min}}^*}{2L(1+\varepsilon)}\frac{d}{n}\varepsilon_2^2\right)
        \]
without the parameter $\varepsilon_0$ that is used in the definition of the active constraints.                
This means that the random projection allow us to ignore the boundary of the constraints,
while the convergence rate becomes $O(\frac{n}{d}\max(\varepsilon_1^{-2},\varepsilon_2^{-2}))$. 
Our methods become more efficient when calculating the gradient $\nabla f(x_k)$ is difficult and require the use of finite difference or forward-mode automatic differentiation.
We can obtain $M_k^\top \fgrad $ with $O(d)$ function evaluations. On the other hand, calculating the full gradients $\fgrad$ requires $O(n)$ function evaluations and this calculation is time-consuming when $n$ is large.   
Hence, if the computational complexity per iteration is dominated by the  gradient calculation, the total complexity is reduced compared with that of the deterministic algorithm.
\end{remark}
\begin{remark}
	We can prove convergence of the deterministic version of our algorithm (i.e., $M_k = I$) by the same argument in Section~\ref{section:Algorithm for linear inequality constraints}. However, the iteration complexity becomes $O(\max{ (\max{(\varepsilon_1^{-2}, \varepsilon_2^{-2})}, \varepsilon_0^{-1}\max{(\varepsilon_1^{-2}, \varepsilon_2^{-2})} ) })$ and we cannot ignore the parameter $\varepsilon_0$.
	When calculating the gradient $\nabla f(x_k)$ is difficult, the time complexity of the deterministic version to reach an approximate KKT point becomes $O(n) \times O(\max{ (\max{(\varepsilon_1^{-2}, \varepsilon_2^{-2})}, \varepsilon_0^{-1}\max{(\varepsilon_1^{-2}, \varepsilon_2^{-2})} ) })$, which is worse than ours with randomness ($O(d) \times O(\frac{n}{d}\max(\varepsilon_1^{-2},\varepsilon_2^{-2}))$).
\end{remark}
Here, we evaluate the computational complexity per iteration of our algorithm.
	Let $T_{value}, T_{grad}$ be the computational complexities of evaluating the function value and the gradient of $f,g_i$. 
	Our method requires $O\left(dn\vert \mathcal{A}_k\vert + m(T_{value}+T_{grad}) + mT_{value} \left\vert \log{\left(\frac{d}{\varepsilon_0 n}\right)}\right\vert \right)$ complexity.
	The first term $O(dn\vert\mathcal{A}_k\vert)$ comes from calculating $M_k^\top G_k$. Using automatic differentiation, we can reduce the complexity to $O(\vert\mathcal{A}_k\vert T_{grad})$. 
	The second term $O(m(T_{grad}+T_{value}))$ comes from calculating the active sets $\mathcal{A}_k$. The last term comes from the while-loop in the proposed method.
	From Propositions~\ref{proposition:Feasibility with linear inequality} and \ref{proposition:Feasibility with linear inequality2}, if $h\beta^k \le O(\min(\frac{n^2\varepsilon_0}{d},\frac{n\varepsilon_0}{d}))$ holds, the while-loop will terminate. %\COMM{AT}{I erased the max iteration number $K$ from Algo 1, so do we need to modify here?}
        Then, we check the feasibility at most $O\left(\left\vert \log{\left(\frac{d}{\varepsilon_0 n}\right)}\right\vert\right)$ times.
The total  computational complexity of our algorithm is estimated by multiplying the iteration complexity in Theorem~\ref{theorem:global convergence} and the above complexity per iteration.

\section{Algorithm for nonlinear inequality constraints}{\label{section:Algorithm for nonlinear inequality constraints}}
In this section, we extend the application of the randomized subspace gradient method to nonlinear constrained problems. The algorithm is summarized in Algorithm~\ref{Algorithm:Third Algorithm}.

\begin{figure}[H]
	\begin{algorithm}[H]
		\caption{Randomized Subspace Gradient for Nonlinear Constrained Problems (RSC-NC)}
		\label{Algorithm:Third Algorithm}
		\begin{algorithmic}
			\State INPUT: $x_0$ (a feasible point), step size $h>0$, reduced dimension $d$, $\varepsilon_0, \delta_1, \varepsilon_2 > 0, 1>\beta>0, \{\mu_k\}_k>0$.
			\For{$k=0,1,2,...$}
			\State $\alpha_k = h$
			\State Set $M_k = \frac{1}{n} P_k^\top$ by sampling a random matrix $P_k \in \mathbb{R}^{d\times n} \sim \No$.
			\State Set $G_k = (\nabla g_i(x_k))_{i \in \mathcal{A}_k} \in \mathbb{R}^{n \times \vert \mathcal{A}_k \vert}$ and
                         $s_k = \left(\|M_k^\top \nabla g_i(x_k)\|\right)_{i\in \mathcal{A}_k}\in \mathbb{R}^{\vert\mathcal{A}_k\vert}$ using $\mathcal{A}_k$ defined in \eqref{eq:active_set}.
			\State Compute
			\begin{align}
				\label{eq:lambda in RSG-NC}
				\notag
				&\bar{\lambda}^{(k)} = -\left(\left(G_k-\mu_k\frac{\fgrad}{\|M_k^\top \fgrad\|}s_k^\top \right)^\top M_kM_k^\top G_k \right)^{-1}\\
				&\hspace{50pt} \left(G_k-\mu_k\frac{\fgrad}{\|M_k^\top \fgrad\|}s_k^\top\right)^\top M_kM_k^\top \fgrad
			\end{align}
			and 
			\begin{equation}
				\label{eq:direction1 in RSG-NC}
				d_k = - M_k^\top \left(\fgrad+G_k\bar{\lambda}^{(k)}\right).
			\end{equation}		
                        \If{$\|d_k\|< \delta_1$}
			\State Compute
			\begin{equation}
				\label{def_lambda_NC}
                                \lambda^{(k)} = -(G_k^\top M_k M_k^\top G_k)^{-1}G_k^\top M_k M_k \fgrad.
                        \end{equation}
			\If{$\min_{i}\lambda^{(k)}_i \geq -\varepsilon_2$}
			\State terminate.
			\EndIf
			\State Compute
			\begin{equation}	
				\label{eq:direction2 in RSG-NC}
				d_k = - \frac{\varepsilon_2 d}{n}M_k^\top G_k (G_k^\top M_k M_k^\top G_k)^{-1}\bar{d}^{(k)}
		        \end{equation}
using 
\begin{empheq}[left = {\bar{d}^{(k)} = \empheqlbrace \,}]{alignat = 2}
            & \mathbf{1} & \quad& \text {if } -\mathbf{1}^\top \lambda^{(k)}\ge \frac{\varepsilon_2}{2}, \label{eq:d-bar1 in direction2 of RSG-NC} \\
            & \chi_{-}(\lambda^{(k)}) + \frac{\sum_{\lambda_j^{(k)} \le 0 }-\lambda_j^{(k)}}{2\sum_{\lambda_j^{(k)}>0} \lambda_j^{(k)}}\chi_{+}(\lambda^{(k)})      & & \text {otherwise. }  \label{eq:d-bar2 in direction2 of RSG-NC}
\end{empheq}
			\EndIf
			\While{$x_k+\alpha_k M_kd_k$ is not feasible}
			\State $\alpha_k\leftarrow \beta \alpha_k$
			\EndWhile
			\State $x_{k+1}\leftarrow x_k+\alpha_k M_kd_k$
			\EndFor
		\end{algorithmic}
	\end{algorithm}
\end{figure}

\subsection{Outline of our algorithm: RSG-NC}
\paragraph{Update of Iterates}
Algorithm~\ref{Algorithm:Third Algorithm} calculates the sequence $\{x_k\}$ by using the step size $\alpha_k$ and the Lagrange multiplier $\bar{\lambda}^{(k)} \in \mathbb{R}^{\vert \mathcal{A}_k \vert}$ corresponding to the constraints in $\mathcal{A}_k$ as 
\begin{align}
	\notag
	x_{k+1} & = x_k - \alpha_k M_kM_k^\top (\fgrad + G_k\bar{\lambda}^{(k)})\\
	\notag
					& = x_k - \alpha_k M_k \left(I - M_k^\top G_k \left(\left(G_k-\mu_k\frac{\fgrad}{\|M_k^\top \fgrad\|}s_k^\top \right)^\top M_kM_k^\top G_k \right)^{-1}\right.\\
					&\hspace{70pt} \left.\left(G_k-\mu_k\frac{\fgrad}{\|M_k^\top \fgrad\|}s_k^\top\right)^\top M_k\right)M_k^\top \fgrad
	\label{eq:iterate1 in RSG-NC}
\end{align}
or
\begin{equation}
	\label{eq:iterate2 in RSG-NC}
	x_{k+1} = x_k - \frac{\varepsilon_2 d}{n}\alpha_k M_k M_k^\top G_k (G_k^\top M_k M_k^\top G_k)^{-1} \bar{d}^{(k)}.
\end{equation}

The search directions $x_{k+1} - x_{k}$ in (\ref{eq:iterate1 in RSG-NC}) and (\ref{eq:iterate2 in RSG-NC}) 
are descent directions with high probability as shown in Propositions~\ref{Proposition:dec_obj_extalgo} and \ref{Proposition:dec_obj_extalgo_2} later in Section~\ref{subsection:Decrease of objective value of RSG-NC}.
Propositions~\ref{proposition:feasibility in algorithm3 1} and \ref{proposition:feasibility in algorithm3 2} in Section~\ref{subsection:Feasibility of RSG-NC} ensure that if $\alpha_k$ are chosen to be less than some threshold, all points in $\{x_k\}$ are feasible with high probability. Hence, the while-loop reducing the value of $\alpha_k$ stops after a fixed number of iterations.  
RSG-NC first computes (\ref{eq:iterate1 in RSG-NC}) and if the search direction is small enough and the Lagrange multiplier $\lambda^{(k)}$, defined by (\ref{def_lambda}), is not nonnegative, it uses (\ref{eq:iterate2 in RSG-NC}) for the next iterate $x_{k+1}$.
The direction $d_k$ of (\ref{eq:direction1 in RSG-NC}) is identical to the one (\ref{def_dk_1}) in Algorithm~\ref{Secondary Algorithm} when $\mu_k = 0$.
Here, we define the following matrices,
\begin{align}
  R'_k := & M_k^\top G_k\left(\left(G_k-\mu_k\frac{\fgrad}{\|M_k^\top \fgrad\|}s_k^\top \right)^\top M_kM_k^\top G_k \right)^{-1} \label{defR'}\\
  &  \hspace{100pt} \left(G_k-\mu_k\frac{\fgrad}{\|M_k^\top \fgrad\|}s_k^\top\right)^\top M_k, \notag \\
  R_k:=& I -R'_k, \notag
\end{align}
and
\begin{equation}
	\label{def:yk}
	y_k := \frac{\varepsilon_2 d}{n}(G_k^\top M_k M_k^\top G_k)^{-1} \bar{d}^{(k)}.
\end{equation}
We can verify that $R_k$ and $R'_k$ are projection matrices; hence, $R_k^2=R_k$ and ${R'}_k^2=R'_k$.

We can therefore rewrite the next iterate (\ref{eq:iterate1 in RSG-NC}) as
\[
  x_{k+1}=x_k - \alpha_k M_kR_k M_k^\top \fgrad
\]
and (\ref{eq:iterate2 in RSG-NC}) as
\[
x_{k+1}=x_k - \alpha_k M_k M_k^\top G_k y_k.
\]
When $\|R_kM_k^\top \fgrad\| = \|M_k^\top (\fgrad + G_k \bar{\lambda}^{(k)})\|$ is small enough, Lemma~\ref{lemma:relathonship berweeen Z_k and R_k} shows that, by choosing a specific value of $\{\mu_k\}$, $\|Z_k M_k^\top \fgrad\| = \|M_k^\top (\fgrad + G_k \lambda^{(k)})\|$ is also small. We recall here that $Z_k$ is defined in (\ref{def:Z_k}).
Therefore, we will check whether $x_k$ is a KKT point or not with the Lagrange multiplier $\lambda^{(k)}$.
As will be shown later in Theorem~\ref{theorem:global convergence of RSG-NC}, when given inputs $d,\varepsilon_0,\varepsilon_2,\delta_1 = \sqrt{\frac{d(1-\varepsilon)}{n^2}}\varepsilon_1$ and $\mu_k = \frac{1}{2\sqrt{s_k^\top (G_k^\top M_k M_k^\top G_k)^{-1}s_k}}$,
Algorithm~\ref{Algorithm:Third Algorithm} generates an output $x_k$ that is guaranteed to be an $(\varepsilon_1,\varepsilon_2,O(\varepsilon_0))$-KKT point.

\subsection{Decrease in objective value}
\label{subsection:Decrease of objective value of RSG-NC}
We will make the following assumptions.
\begin{assume}{\label{Assumption:Constraints}}
  \begin{enumerate}[label=(\roman*)]
		\item All constraints, $g_i$ for $\forall i$, are $L_g$-smooth on the feasible set.{\label{Assumption:const_Lg-smooth}}
		\item There exists an interior point $x$, i.e., $g_i(x)<0$ for each $i$. \label{Assumption:const_int}
	\end{enumerate}
\end{assume}
Assumption~\ref{Assumption:Constraints}\ref{Assumption:const_Lg-smooth} is satisfied when Assumption~\ref{Assumption:Objective}{\ref{Level Set is compact or feasible set is compact}} is satisfied and $g_i$ are twice continuously differentiable functions. 
Assumptions~\ref{Assumption:Objective}\ref{Level Set is compact or feasible set is compact} and \ref{Assumption:Constraints}\ref{Assumption:const_Lg-smooth} imply that $\|\nabla g_i\|$ are continuous and bounded on the feasible set ($\|\nabla g_i(x)\|\le U_g$).
		
\begin{assume}\label{Assumption:minimum of Constraints gradient}{\label{Assumption:maximum of Constraints}}
The following problems for all $i$:
\begin{eqnarray}
		\min_{x} \|\nabla g_i(x)\|   \quad \mathrm{s.t.}\;  f(x)\le f(x_0),~[-g_i(x)]_+ \le \varepsilon_0 \|\nabla g_i(x) \|, ~x\in \mathcal{C}
		\label{Problem:lower bound of gradient}
	\end{eqnarray}
	and problems for all $i$:
	\begin{eqnarray}
		\max_{x} g_i(x) \quad \mathrm{s.t.}\;  f(x)\le f(x_0),~ \varepsilon_0 \|\nabla g_i(x) \| \le [-g_i(x)]_+ , ~x\in \mathcal{C}
		\label{Problem:upper bound  of function value} 
	\end{eqnarray}
	have nonzero optimal values. 
\end{assume}
This assumption means that $\|\nabla g_i(x)\|$ is not zero when $x$ is close to the boundary $g_i(x) = 0$ and $g_i(x)$ is not zero when 
$x$ is far from the boundary. If the problems~(\ref{Problem:lower bound of gradient}) and (\ref{Problem:upper bound  of function value}) have no solutions, we set the optimal values to $\infty$ and $-\infty$, respectively. 
We set $l_g^{\varepsilon_0}~(>0)$ for the minimum of all optimal values of \eqref{Problem:lower bound of gradient} over all $i$ 
and $u_g^{\varepsilon_0}~(<0)$ for the maximum of all optimal values of  \eqref{Problem:upper bound  of function value} over all $i$.
Let $g_*~(<0)$ denote the minimum of all optimal values of the following problems over $i$:
\[
	\min_{x} g_i(x) \quad \mathrm{s.t.}\; f(x)\le f(x_0), ~x\in \mathcal{C}.
\]
$g_*$ is bounded from Assumptions~\ref{Assumption:Objective}\ref{Level Set is compact or feasible set is compact} and \ref{Assumption:Constraints}\ref{Assumption:const_Lg-smooth}.
The following lemma implies that if the constraints are all convex, Assumption~\ref{Assumption:maximum of Constraints}
is not necessary because it is proved to hold from more general standard assumptions as follows.

\begin{lemma}\label{Lemma:upper bound out boundary} 
  Suppose that Assumptions~\ref{Assumption:Objective}\ref{Level Set is compact or feasible set is compact} and \ref{Assumption:Constraints} hold and that $g_i$ are all convex.
  Then, Assumption~\ref{Assumption:maximum of Constraints} holds.
\end{lemma}	
\begin{proof}
	If the feasible set is not empty,  from Assumptions~\ref{Assumption:Objective}\ref{Level Set is compact or feasible set is compact} and \ref{Assumption:Constraints}\ref{Assumption:const_Lg-smooth}, there exist optimal solutions for (\ref{Problem:lower bound of gradient}) and (\ref{Problem:upper bound  of function value}), respectively.
	Suppose that there exists an optimal solution $\hat{x}^*$ such that $\|\nabla g_i(\hat{x}^*)\|=0$ for (\ref{Problem:lower bound of gradient}).  
	Since $ [-g_i(\hat{x}^*)]_+ \le \varepsilon_0\|\nabla g_i(\hat{x}^*)\|$ holds, $g_i(\hat{x}^*)=0$. Similarly 
	under the assumption that there exists $\tilde{x}^*$ such that $g_i(\tilde{x}^*)=0$ for (\ref{Problem:upper bound of function value}),  
	we see that $\varepsilon_0 \|\nabla g_i(\tilde{x}^*)\|\le [- g_i(\tilde{x}^*)]_+$ holds and obtain $\|\nabla g_i(\tilde{x}^*)\|=0$.
	By the convexity of $g_i$, for all $x$ and $x^* \in \{\hat{x}^*,~\tilde{x}^*\}$
	\[
	g_i(x)\ge g_i(x^*)+\langle \nabla g_i(x^*),x-x^*\rangle = 0
	\]
	holds for any $i$; hence, $g_i(x)\ge 0$. This contradicts Assumption~\ref{Assumption:Constraints}\ref{Assumption:const_int}.
\end{proof}

Next, we prove that the update direction in Algorithm~\ref{Algorithm:Third Algorithm} is a descent direction for a specific value of $\{\mu_k\}_k$.
Let us consider the orthogonal projection of $ M_k^\top \nabla f(x_k)$ into the image of $M_k^\top G_k$:
	\begin{equation}{\label{eq:projection gradient to Constraints subspace}}
		M_k^\top \nabla f(x_k)= -M_k^\top G_k\lambda^{(k)} + Z_k M_k^\top \fgrad,
	\end{equation}
        where $\lambda^{(k)}$ is defined by (\ref{def_lambda}).
\begin{lemma}
  {\label{lemma:relathonship berweeen Z_k and R_k}}
  Suppose that Assumption~\ref{Assumption:existence of inverse} holds. 
	Let $\mu_k = \frac{1}{2\sqrt{s_k^\top (G_k^\top M_k M_k^\top G_k)^{-1} s_k}}$.
	Then,
        \begin{itemize}
        \item
          $\left(\left(G_k-\mu_k\frac{\fgrad}{\|M_k^\top \fgrad\|}s_k^\top \right)^\top M_kM_k^\top G_k \right)^{-1}$ exists,
\item the following inequalities hold:
	  \begin{equation}{\label{eq:eval_1}}
		\frac23\|Z_k M_k^\top \fgrad \|^2 \le \fgrad^\top M_k R_k M_k^\top \fgrad \le 2 \|Z_k M_k^\top \fgrad\|^2,
	\end{equation}
	\begin{equation}{\label{eq:eval_2}}
		\|Z_k M_k^\top \fgrad\|^2\le \|R_k M_k^\top \fgrad\|^2 \le 2 \|Z_k M_k^\top \fgrad\|^2, and
	\end{equation}
        \item
          $d_k = - M_kR_kM_k^\top \fgrad$ is a descent direction.
        \end{itemize}
\end{lemma}
\begin{proof}
  Let $\nu  = \frac{\mu_ks_k}{\|M_k^\top \fgrad\|}$; we will show that $\left(G_k-\fgrad\nu^\top \right)^\top M_kM_k^\top G_k$ is non-singular.
Here,
	\begin{alignat}{2}
		\notag
		\left(G_k-\fgrad\nu^\top \right)^\top M_kM_k^\top G_k 
		=& G_k^\top M_kM_k^\top G_k -\nu\fgrad^\top M_kM_k^\top G_k\\
		\notag
		=& G_k^\top M_kM_k^\top G_k + \nu\lambda^{(k)\top} G_k^\top M_kM_k^\top G_k \\
		=& (I + \nu\lambda^{(k)\top})(G_k^\top M_kM_k^\top G_k), \label{matdecomp}
	\end{alignat}
	where the second equality follows from (\ref{eq:projection gradient to Constraints subspace}).
	We will confirm that $I + \nu\lambda^{(k)\top}$ is invertible: notice that $\nu$ is an eigenvector of $I + \nu\lambda^{(k)\top}$ whose corresponding eigenvalue is $1+\nu^\top \lambda^{(k)} $.
	All the other eigenvectors are orthogonal to $\lambda^{(k)}$ and their corresponding eigenvalues are equal to one.
	From the definition of $\nu,\mu_k,\lambda^{(k)}$, 
	\begin{flalign*}
		\notag
		\vert \nu^\top \lambda^{(k)}\vert
		&= \frac{\mu_k}{\|M_k^\top \fgrad\|}\vert s_k^\top (G_k^\top M_k M_k^\top G_k)^{-1}G_k^\top M_k M_k\fgrad\vert &\\
		\notag
		&\le \frac{\mu_k}{\|M_k^\top \fgrad\|} \|M_k^\top G_k (G_k^\top M_k M_k^\top G_k)^{-1}s_k\|\|M_k^\top \fgrad\| \notag\\
		&= \mu_k \sqrt{s_k^\top (G_k^\top M_k M_k^\top G_k)^{-1}s_k} = \frac12
	\end{flalign*}
	holds, as $\|M_k^\top G_k (G_k^\top M_k M_k^\top G_k)^{-1}s_k\|^2=s_k^\top (G_k^\top M_k M_k^\top G_k)^{-1}s_k$.
	Hence, we obtain $1/2 \le \nu^\top \lambda^{(k)} +1\le 3/2$ which implies that $(I + \nu\lambda^{(k)\top})$ is invertible as all of its eigenvalues are non-zero. From Assumption~\ref{Assumption:existence of inverse}, $G_k^\top M_k M_k^\top G_k$ is also invertible; thus, from \eqref{matdecomp}, 
	$\left(\left(G_k-\fgrad\nu^\top \right)^\top M_kM_k^\top G_k\right)^{-1}$ exists.

	Next, we calculate $\fgrad^\top M_k R_k M_k^\top \fgrad$ and $\|R_k M_k^\top \fgrad\|^2$.
	From (\ref{eq:projection gradient to Constraints subspace}) and the definition of the orthogonal projection, we obtain 
	\begin{equation}{\label{eq:orthogonal projection norm}}
		\|M_k^\top \fgrad\|^2 = \|M_k ^\top G_k \lambda^{(k)}\|^2 + \|Z_k M_k^\top \fgrad\|^2.
	\end{equation}
	In order to project $M_k^\top \fgrad$ by $R_k$, first of all, we need to rewrite $R'_k$ in \eqref{defR'} as
	\begin{alignat*}{2}
		\notag
		R_k'& = M_k^\top G_k\left(\left(G_k-\fgrad\nu^\top \right)^\top M_kM_k^\top G_k \right)^{-1}\left(G_k-\fgrad\nu^\top \right)^\top M_k & \\
		\notag
		& = M_k^\top G_k(G_k^\top M_kM_k^\top G_k)^{-1}(I + \nu\lambda^{(k)\top})^{-1}\left(G_k-\fgrad\nu^\top \right)^\top M_k
	\end{alignat*}
        using \eqref{matdecomp}.
	We project $M_k^\top \fgrad$ by $R_k'$ and obtain
	\begin{alignat*}{2}
		\notag
		& R_k' M_k^\top \fgrad \\
		\notag
		& = M_k^\top G_k (G_k^\top M_k M_k^\top G_k)^{-1}(I + \nu\lambda^{(k)\top} )^{-1} (G_k^\top M _k M_k^\top \fgrad  - \|M_k^\top \fgrad\|^2 \nu ) &\\
		\notag
		& = M_k^\top G_k (G_k^\top M_k M_k^\top G_k)^{-1}(I + \nu\lambda^{(k)\top} )^{-1} (-G_k^\top M _k M_k^\top G_k\lambda^{(k)}  - \|M_k^\top \fgrad\|^2 \nu ) \\
		\notag
		& = -M_k^\top G_k (G_k^\top M_k M_k^\top G_k)^{-1}(I + \nu\lambda^{(k)\top} )^{-1}\\
		& \hspace{150pt} ((I+ \nu\lambda^{(k)\top})G_k^\top M _k M_k^\top G_k\lambda^{(k)}  + \|Z_kM_k^\top \fgrad\|^2 \nu )\\
		& = -M_k^\top G_k \lambda^{(k)} - \|Z_kM_k^\top \fgrad\|^2 M_k^\top G_k (G_k^\top M_k M_k^\top G_k)^{-1}(I + \nu\lambda^{(k)\top} )^{-1} \nu. &
	\end{alignat*}
	The second equality follows from (\ref{eq:projection gradient to Constraints subspace}) and the third equality comes from (\ref{eq:orthogonal projection norm}).
	Note that $\nu$ is a eigenvector of $(I + \nu\lambda^{(k)\top})^{-1}$, and
	\[
	(I + \nu\lambda^{(k)\top})^{-1}\nu = \frac{\nu}{1+\lambda^{(k)\top} \nu}
	\]
	holds. These relations leads us to
	\begin{equation*}
		R_k' M_k^\top \fgrad = -M_k^\top G_k \lambda^{(k)} - \frac{\|Z_kM_k^\top \fgrad\|^2}{1 + \lambda^{(k)\top} \nu} M_k^\top G_k (G_k^\top M_k M_k^\top G_k)^{-1}\nu.
	\end{equation*}
	From this equation, we obtain
	\begin{alignat}{2}
		\notag
		&\fgrad^\top M_k R'_k M_k^\top \fgrad \\
		\notag
		&= -\fgrad^\top M_kM_k^\top G_k \lambda^{(k)} - \frac{\|Z_kM_k^\top \fgrad\|^2}{1 + \lambda^{(k)\top} \nu} \fgrad^\top M_kM_k^\top G_k (G_k^\top M_k M_k^\top G_k)^{-1}\nu\\
		{\label{eq:oblique}}
		&= \| M_k^\top G_k \lambda^{(k)} \|^2 +\frac{ \nu^\top \lambda^{(k)} }{ \nu^\top \lambda^{(k)}+1} \|Z_kM_k^\top \fgrad\|^2, 
	\end{alignat}
	where the last equality follows from (\ref{eq:projection gradient to Constraints subspace}).
	We also obtain
	\begin{align}{\label{eq:norm_oblique}}
		\notag
		&\|R'_kM_k^\top \fgrad\|^2\\
		&= \| M_k^\top G_k \lambda^{(k)} \|^2 
		+\frac{ 2\nu^\top \lambda^{(k)} }{ \nu^\top \lambda^{(k)}+1} \|Z_kM_k^\top \fgrad\|^2
		+\frac{\|Z_kM_k^\top \fgrad\|^4}{(\nu^\top \lambda^{(k)}+1)^2}\nu^\top (G_k^\top M_k M_k^\top G_k)^{-1} \nu.
	\end{align}
	Recalling that $R_k = I - R_k'$, we have
	\begin{alignat}{2}
		\notag
		&\fgrad^\top M_k R_k M_k^\top \fgrad \\
		\notag
		&= \|M_k^\top \fgrad\|^2 - \fgrad^\top M_k R'_k M_k^\top \fgrad\\
		\notag
		&= \|M_k ^\top G_k \lambda^{(k)}\|^2 + \|Z_kM_k^\top \fgrad\|^2 - \| M_k^\top G_k \lambda^{(k)} \|^2 - \frac{ \nu^\top \lambda^{(k)} }{ \nu^\top \lambda^{(k)}+1} \|Z_kM_k^\top \fgrad\|^2  \\
		{\label{eq:oblique_2}}
		&= \frac{1}{ \nu^\top \lambda^{(k)}+1} \|Z_kM_k^\top \fgrad\|^2.
	\end{alignat}
	The second equality follows from (\ref{eq:orthogonal projection norm}) and (\ref{eq:oblique}).
	Using  (\ref{eq:oblique_2}) and $1/2 \le \nu^\top \lambda^{(k)} +1\le 3/2$, we find that
	\begin{equation*}
		\frac23\|Z_kM_k^\top \fgrad\|^2 \le \fgrad^\top M_k R_k M_k^\top \fgrad \le 2 \|Z_kM_k^\top \fgrad\|^2.
	\end{equation*}        
	Using $R_k = I - R_k'$, we also have
	\begin{alignat}{2}
		\notag
		&\|R_k M_k^\top \fgrad\|^2& \\
		\notag
		&= \|M_k^\top \fgrad\|^2 - 2 \fgrad^\top M_k R'_k M_k^\top \fgrad + \|R'_kM_k^\top \fgrad\|^2\\
		{\label{eq:norm_oblique_2}}
		&= \|Z_kM_k^\top \fgrad\|^2 + \frac{\|Z_kM_k^\top \fgrad\|^4}{(\nu^\top \lambda^{(k)}+1)^2}\nu^\top (G_k^\top M_k M_k^\top G_k)^{-1} \nu.
	\end{alignat}
	The last equality follows from (\ref{eq:orthogonal projection norm}),(\ref{eq:oblique})\;and\;(\ref{eq:norm_oblique}).
	We can now evaluate the last term of (\ref{eq:norm_oblique_2}) as 
	\begin{alignat*}{2}
		&\|Z_kM_k^\top \fgrad\|^2\nu^\top (G_k^\top M_k M_k^\top G_k)^{-1} \nu&\\
		&= \frac{\mu_k^2\|Z_kM_k^\top \fgrad\|^2}{\|M_k^\top\fgrad\|^2} s_k^\top (G_k^\top M_k M_k^\top G_k)^{-1} s_k\\
		&= \frac14\frac{\|Z_kM_k^\top \fgrad\|^2}{\|M_k^\top\fgrad\|^2} \\
		&\le \frac14.
	\end{alignat*}
	The first equality follows from definition of $\nu$, the second equality follows from definition of $\mu_k$, and 
	the last inequality follows from (\ref{eq:orthogonal projection norm}).
	Finally, by using (\ref{eq:norm_oblique_2}) and the above upper bound, we obtain
	\begin{equation*}
		\|Z_kM_k^\top \fgrad\|^2\le\|R_k M_k^\top \fgrad\|^2 \le 2 \|Z_kM_k^\top \fgrad\|^2.
	\end{equation*}

  Lastly, we can easily confirm from \eqref{eq:oblique_2} that  $d_k = - M_kR_kM_k^\top \fgrad$ is a descent direction.
\end{proof}
\begin{proposition}\label{Proposition:dec_obj_extalgo}
	Let $\mu_k = \frac{1}{2\sqrt{s_k^\top (G_k^\top M_k M_k^\top G_k)^{-1}s_k}}$. Under Assumptions~\ref{Assumption:Objective}\ref{Objective is L-smooth} and \ref{Assumption:existence of inverse}, when $d_k = - M_k^\top (\fgrad + G_k \bar{\lambda}^{(k)})$ of (\ref{eq:direction1 in RSG-NC}) in Algorithm~\ref{Algorithm:Third Algorithm}, we have
	\[
	f(x_{k+1})\le f(x_k) - \left(\frac23 \alpha_k\ -\frac{\alpha_k^2 L(1+\varepsilon)}{n} \right)\|Z_kM_k^\top \fgrad\|^2
	\]
	with probability at least $1-2\exp{(-C_0 \varepsilon^2 n)}$. 
\end{proposition}
\begin{proof}
	Using the same argument as in Proposition~\ref{Proposition:dec_obj}, from the $L$-smoothness (\ref{eq:L-smooth}) of the objective function $f$, Lemma~\ref{lemma:Johnson}, and $x_{k+1}-x_{k} = -\alpha_k M_k R_k M_k^\top \fgrad$ from (\ref{eq:iterate1 in RSG-NC}), we obtain
	\begin{alignat*}{2}\label{eq:L-smooth inequality}
		\notag
		&f(x_{k+1}) - f(x_k)\\
		\notag
		&\le- \alpha_k\fgrad^\top M_kR_k M_k^\top \fgrad+\frac{\alpha_k^2L}{2}\| M_kR_k M_k^\top \fgrad\|^2 \\
		&\le - \alpha_k\fgrad^\top M_kR_k M_k^\top \fgrad+\frac{\alpha_k^2L(1+\varepsilon)}{2n}\| R_k M_k^\top \fgrad\|^2.
	\end{alignat*}
	The last inequality follows form Lemma~\ref{lemma:Johnson}.
	Combining (\ref{eq:eval_1}) and (\ref{eq:eval_2}), we find that
	\begin{equation*}{\label{eq:finalform}}
		f(x_{k+1}) - f(x_k) \le -\left(\frac23 \alpha_k\ -\frac{\alpha_k^2 L(1+\varepsilon)}{n} \right)\|Z_kM_k^\top \fgrad\|^2
	\end{equation*}
	holds with probability at least $1-2\exp{(-C_0 \varepsilon^2 n)}$.
\end{proof}

\begin{proposition}{\label{Proposition:dec_obj_extalgo_2}}
	Suppose that Assumptions~{\ref{Assumption:Objective}\ref{Objective is L-smooth}} and \ref{Assumption:existence of inverse} hold and let $d_k= -\frac{\varepsilon_2d}{n} M_k^\top G_k(G_k^\top M_k M_k^\top G_k)^{-1}\bar{d}^{(k)}$, i.e.,  (\ref{eq:direction2 in RSG-NC}) with $\bar{d}^{(k)}$ defined by (\ref{eq:d-bar1 in direction2 of RSG-NC}) or (\ref{eq:d-bar2 in direction2 of RSG-NC}) in Algorithm~\ref{Algorithm:Third Algorithm} and $\min_{i}\lambda^{(k)}_i < - \varepsilon_2$.
	Then,
	\begin{equation*}
		f(x_{k+1})- f(x_k)
		\le -\alpha_k\left(1 - \frac{(1+\varepsilon)L\alpha_k\vert \mathcal{A}_k\vert }{(1-\varepsilon)\lambda_{\mathrm{min}}(G_k^\top G_k)}\right)\frac{\varepsilon_2^2d}{2n}
	\end{equation*} 
	holds with probability at least $1-2\exp{(-C_0\varepsilon^2n)}-2\exp{(-C_0\varepsilon^2d)}$.
\end{proposition}
\begin{proof}
	First, we show that $M_kd_k$ is a descent direction for both $\bar{d}^{(k)}$ defined by (\ref{eq:d-bar1 in direction2 of RSG-NC}) and by (\ref{eq:d-bar2 in direction2 of RSG-NC}). We have
	\begin{eqnarray}{\label{eq:Algorithm3 second descent direction}}
		&\fgrad^\top M_kd_k
		&= -\frac{\varepsilon_2 d}{n}\fgrad^\top M_kM_k^\top G_k\left(G_k^\top M_kM_k^\top G_k\right)^{-1}\bar{d}^{(k)} \notag\\
		&&=  \frac{\varepsilon_2 d}{n}\lambda^{(k)\top} \bar{d}^{(k)} 
	\end{eqnarray}
  using the definition \eqref{def_lambda_NC} of $\lambda^{(k)}$.
	When $-\mathbf{1}^\top \lambda^{(k)} \ge \frac{\varepsilon_2}{2}$ and $\bar{d}^{(k)} = \mathbf{1}$, i.e., (\ref{eq:d-bar1 in direction2 of RSG-NC}), (\ref{eq:Algorithm3 second descent direction}) gives
	\begin{equation}{\label{eq:Algorithm3 descent direction1}}
		\fgrad^\top M_kd_k \le -\frac{\varepsilon_2^2 d}{2n}.  %\ge \frac{\varepsilon_2^2 d}{2n}
	\end{equation}
	Moreover, when $\bar{d}_i^{(k)} = \chi_{-}(\lambda_i^{(k)}) + \frac{\sum_{\lambda_j^{(k)} \le0 }-\lambda_j^{(k)}}{2\sum_{\lambda_j^{(k)}>0} \lambda_j^{(k)}}\chi_{+}(\lambda_i^{(k)} )$, i.e., (\ref{eq:d-bar2 in direction2 of RSG-NC}), we have
	\begin{alignat}{2}{\label{eq:Algorithm3 descent direction2}}
		\notag
		&\fgrad^\top M_kd_k&\\
		\notag
		&= \frac{\varepsilon_2 d}{n} \left(\sum_{\lambda_j^{(k)} \le 0} \lambda_j^{(k)} + \sum_{\lambda_j^{(k)} > 0} \lambda_j^{(k)}\frac{\sum_{\lambda_i^{(k)} \le 0 }-\lambda_i^{(k)}}{2\sum_{\lambda_i^{(k)}>0} \lambda_i^{(k)}}\right) \\
		\notag
		& = \frac{\varepsilon_2 d}{n} \left(\sum_{\lambda_j^{(k)} \le 0} \lambda_j^{(k)}  - \frac12\sum_{\lambda_j^{(k)} \le 0} \lambda_j^{(k)} \right)\\
		\notag
		&= \frac{\varepsilon_2 d}{2n}\sum_{\lambda_j^{(k)} \le 0} \lambda_j^{(k)} \\
		&< -\frac{\varepsilon^2_2 d}{2n}.
                %		&> \frac{\varepsilon^2_2 d}{2n}.
	\end{alignat}
	The first equality follows from (\ref{eq:Algorithm3 second descent direction}) and definition of $\bar{d}^{(k)}$. The last inequality follows from $\min_{i}\lambda_i^{(k)}<-\varepsilon_2$.
	Thus, in both cases, $M_kd_k$ is a descent direction.
	Next, we will evaluate the decrease $f(x_{k+1})- f(x_k)$. By using (\ref{eq:L-smooth}) from Assumption~\ref{Assumption:Objective}\ref{Objective is L-smooth} and
	$x_{k+1} - x_{k} = \alpha_k M_kd_k =-\alpha_k M_k M_k ^\top G_k y_k$ from (\ref{eq:iterate2 in RSG-NC}), we have
	\begin{eqnarray}{\label{eq:Algorithm3 L-smooth}}
		f(x_{k+1})\le f(x_k) + \alpha_k \fgrad^\top M_k d_k % - \alpha_k \fgrad^\top M_kM_k^\top G_ky_k
		+ \frac{L}{2}\alpha_k^2 \|M_kM_k^\top G_k y_k\|^2.
	\end{eqnarray}
	We apply Lemma~\ref{lemma:Johnson} to the last term and obtain 
	\begin{alignat*}{2}{\label{eq:Algorithm3 last term evaluate}}
		\notag
		\|M_kM_k^\top G_ky_k\|^2 &\le \frac{1+\varepsilon}{n}\|M_k^\top G_ky_k\|^2\\
		\notag
		&= \frac{\varepsilon_2^2d^2(1+\varepsilon)}{n^3} \bar{d}^{(k)\top} (G_k^\top M_k M_k^\top G_k)^{-1} \bar{d}^{(k)} \\
		&\le \frac{\varepsilon_2^2d^2(1+\varepsilon)}{n^3\lambda_{\mathrm{min}}{(G_k^\top M_k M_k^\top G_k)}}\|\bar{d}^{(k)}\|^2.
	\end{alignat*}
	The first inequality follows from Lemma~\ref{lemma:Johnson} with $M_k = \frac{1}{n}P_k^\top$ and the first equality follows from definition of $y_k$ in (\ref{def:yk}).
	Furthermore, we evaluated $\lambda_{\mathrm{min}}(G_k^\top M_k M_k^\top G_k)$ in (\ref{eq:eig_linear}) with probability at least $1-2\exp{(-C_0\varepsilon^2d)}$.
	Then, we have
	\begin{equation}{\label{eq:norm of MMGy}}
		\|M_kM_k^\top G_ky_k\|^2 
		\le \frac{\varepsilon_2^2d(1+\varepsilon)}{n(1-\varepsilon)\lambda_{\mathrm{min}}{(G_k^\top  G_k)}}\|\bar{d}^{(k)}\|^2.
	\end{equation}
	We can compute an upper bound for $\|\bar{d}^{(k)}\|$, where $\bar{d}^{(k)}$ is defined by (\ref{eq:d-bar1 in direction2 of RSG-NC}) or (\ref{eq:d-bar2 in direction2 of RSG-NC}). When $-\mathbf{1}^\top\lambda^{(k)} \geq \frac{\varepsilon_2}{2}$ and $\bar{d}^{k} = \bm{1}$ from (\ref{eq:d-bar1 in direction2 of RSG-NC}), it is clear that $\|\bar{d}^{(k)}\|  = \|\mathbf{1}\|= \sqrt{\vert\mathcal{A}_k\vert}<\sqrt{d}$ (by Assumption~\ref{Assumption:existence of inverse}\ref{Assumption:reduced dimension is larger than active sets}).
	When $-\mathbf{1}^\top \lambda^{(k)} < \frac{\varepsilon_2}{2}$ and $\bar{d}^{(k)}$ is defined by (\ref{eq:d-bar2 in direction2 of RSG-NC}), we have
	\begin{equation}{\label{eq:lambda eq1}}
		-\sum_{\lambda^{(k)}_i\le 0} \lambda^{(k)}_i < \sum_{\lambda^{(k)}_i> 0} \lambda^{(k)}_i + \frac{\varepsilon_2}{2},
	\end{equation}
	and since $\min_{i}\lambda_i^{(k)}< - \varepsilon_2<0$, we have
	\begin{equation}{\label{eq:lambda eq2}}
		\varepsilon_2 < -\sum_{\lambda^{(k)}_i\le0} \lambda^{(k)}_i.
	\end{equation}
	From (\ref{eq:lambda eq1}) and (\ref{eq:lambda eq2}), we find that
	\begin{equation}{\label{eq:lambda eq3}}
		\frac{\varepsilon_2}{2} < \sum_{\lambda^{(k)}_i> 0} \lambda^{(k)}_i.
	\end{equation}
	These relations leads us to 
	\begin{alignat*}{2}{\label{eq:upper bound of d2}}
		\notag
		\frac{\sum_{\lambda^{(k)}_j \le 0 }-\lambda^{(k)}_j}{2\sum_{\lambda^{(k)}_j>0} \lambda^{(k)}_j}
		&< \frac{\varepsilon_2}{4\sum_{\lambda^{(k)}_j>0} \lambda_j^{(k)}} + \frac{1}{2} \\
		&< \frac{1}{2} + \frac{1}{2}  = 1.
	\end{alignat*}
	The first inequality follows from (\ref{eq:lambda eq1}) and the last inequality follows from (\ref{eq:lambda eq3}).
	Accordingly, we have $\|\bar{d}^{(k)}\| = \sqrt{\sum_{\lambda^{(k)}_j \le 0} 1 + \sum_{\lambda^{(k)}_j >0} \left(\frac{\sum_{\lambda^{(k)}_j <0 }-\lambda^{(k)}_j}{2\sum_{\lambda^{(k)}_j>0} \lambda^{(k)}_j}\right)^2}\le \sqrt{\vert\mathcal{A}_k\vert}$. 
	Thus, $\|\bar{d}^{(k)}\| \le \sqrt{\vert \mathcal{A}_k \vert}$ holds when $\bar{d}^{(k)}$ is defined by (\ref{eq:d-bar1 in direction2 of RSG-NC}) and by (\ref{eq:d-bar2 in direction2 of RSG-NC}).
	Combining (\ref{eq:Algorithm3 descent direction1}), (\ref{eq:Algorithm3 descent direction2}), (\ref{eq:Algorithm3 L-smooth}), (\ref{eq:norm of MMGy}) and $\|\bar{d}^{(k)}\|\le \sqrt{\vert\mathcal{A}_k\vert}$, we obtain the following lower bound of the step size,
	\begin{alignat*}{2}
		\notag
		f(x_{k+1})- f(x_k)
		&\le - \alpha_k \frac{\varepsilon_2^2 d}{2n} + \frac{L}{2}\alpha_k^2 \|M_kM_k^\top G_ky_k\|^2\\
		\notag
		&\le - \alpha_k \frac{\varepsilon_2^2 d}{2n} + \frac{L}{2}\alpha_k^2 \frac{(1+\varepsilon)\varepsilon_2^2d}{n(1-\varepsilon)\lambda_{\mathrm{min}}{(G_k^\top  G_k)}}\|\bar{d}^{(k)}\|^2 \\
		&\le -\alpha_k\left(1 - \frac{(1+\varepsilon)L\alpha_k\vert\mathcal{A}_k\vert}{(1-\varepsilon)\lambda_{\mathrm{min}}(G_k^\top G_k)}\right)\frac{\varepsilon_2^2d}{2n}.
	\end{alignat*}
	The first inequality follows from (\ref{eq:Algorithm3 L-smooth}) and, (\ref{eq:Algorithm3 descent direction1}) or (\ref{eq:Algorithm3 descent direction2}). The second inequality follows from (\ref{eq:norm of MMGy}) and the last inequality from $\|\bar{d}^{(k)}\|\le \sqrt{\vert\mathcal{A}_k\vert}$.
\end{proof}
\subsection{Feasibility}
\label{subsection:Feasibility of RSG-NC}
Here, by utilizing the $L_g$-smoothness of the constraints from Assumption~\ref{Assumption:Constraints}\ref{Assumption:const_Lg-smooth}, we derive conditions on the step size so that the sequence $\{x_k\}$ generated by Algorithm~\ref{Algorithm:Third Algorithm} is feasible.

\begin{proposition}{\label{proposition:feasibility in algorithm3 1}}
	Let $\mu_k = \frac{1}{2\sqrt{s_k^\top (G_k^\top M_k M_k^\top G_k)^{-1} s_k}}$ and assume that $x_k$ is feasible. Furthermore, suppose that Assumptions~\ref{Assumption:Objective}\ref{Level Set is compact or feasible set is compact}, \ref{Assumption:existence of inverse}, \ref{Assumption:Constraints}, and \ref{Assumption:maximum of Constraints} hold. 
	Then, if the step size $\alpha_k$ satisfies
	\[
	0 \le \alpha_k \le \min\left(
	\frac{\sqrt{\lambda_{\mathrm{min}}^*} l_g^{\varepsilon_0}(1-\varepsilon)}{3U_fU_gL_g(1+\varepsilon)^2}\frac{n}{\sqrt{d}},
	\frac{1}{U_f}\sqrt{\frac{n^3}{2d(1+\varepsilon)^2}}\frac{-2u_g^{\varepsilon_0}}{U_g+\sqrt{U_g^2 - 2L_gg_*}}
	\right),
	\]
	when $d_k =- M_kR_kM_k^\top \fgrad $ of (\ref{eq:direction1 in RSG-NC}) in Algorithm~\ref{Algorithm:Third Algorithm}, $x_{k+1}$ is feasible with probability at least $1-2\exp{(-C_0\varepsilon^2n)}-(2\vert\mathcal{A}_k\vert+6)\exp{(-C_0\varepsilon^2d)}$.
\end{proposition}
\begin{proof}
	First, let us consider the active constraints $g_i$ $(i\in \mathcal{A}_k)$.
	Since $g_i$ are $L_g$-smooth from Assumption~\ref{Assumption:Constraints}\ref{Assumption:const_Lg-smooth}, we have
	\begin{equation}
		\label{eq:Lg-smooth of constraints}
		g_i(x_{k+1}) \le g_i(x_k) + \nabla g_i(x_k)^\top (x_{k+1} - x_{k}) + \frac{L_g}{2}\|x_{k+1}- x_k\|^2.
	\end{equation}
	From (\ref{eq:Lg-smooth of constraints}) and $x_{k+1} - x_{k} = -\alpha_k M_k R_k M_k^\top \fgrad $ of (\ref{eq:iterate1 in RSG-NC}), we have
	\begin{alignat}{2}
		\notag
		&g_i(x_{k+1})\le g_i(x_k) - \alpha_k \nabla g_i(x_k)^\top M_kR_k M_k^\top \fgrad +\frac{\alpha_k^2L_g}{2} \|M_kR_k M_k^\top \fgrad\|^2. &\\
\label{Lsmooth_g}
	\end{alignat}
	Note that $\bar{\lambda}^{(k)}$ of \eqref{eq:lambda in RSG-NC} is the solution of 
	\begin{align*}
			\left(G_k - \mu_k \frac{\fgrad s_k^\top}{\|M_k^\top \fgrad\|}\right)^\top M_k (M_k^\top \fgrad + M_k^\top G_k\bar{\lambda}^{(k)}) = 0.
	\end{align*}
Recalling that
\begin{align*}
	R'_k := & M_k^\top G_k\left(\left(G_k-\mu_k\frac{\fgrad}{\|M_k^\top \fgrad\|}s_k^\top \right)^\top M_kM_k^\top G_k \right)^{-1} \\
	&  \hspace{100pt} \left(G_k-\mu_k\frac{\fgrad}{\|M_k^\top \fgrad\|}s_k^\top\right)^\top M_k, \notag \\
	R_k:=& I -R'_k, 
\end{align*}
	we deduce that $R_k M_k^\top \fgrad = M_k^\top \fgrad + M_k^\top G_k \bar{\lambda}^{(k)}$; thus,
	\[
		G_k^\top M_k R_k M_k^\top \fgrad = \mu_k \frac{s_k}{\|M_k^\top \fgrad \|} \fgrad^\top M_k R_k M_k^\top \fgrad,
	\]
	which is equivalent to  
	\begin{equation}
		\label{eq:cond_updaterule}
		\nabla g_i(x_k)^\top M_k R_k M_k^\top \fgrad = \mu_k \frac{\| M_k^\top \nabla g_i(x_k)\|}{\| M_k^\top \fgrad \|} \fgrad^\top M_k R_k M_k^\top \fgrad,
	\end{equation}
	for all $i \in \mathcal{A}_k$. From this equation and \eqref{Lsmooth_g},
        we obtain that for all $i \in \mathcal{A}_k$,
	\begin{alignat*}{2}
	\notag
	 & g_i(x_{k+1})\\ 
	\notag
		&\le g_i(x_k)- \mu_k\alpha_k \frac{\|M_k^\top \nabla g_i(x_k)\|}{\|M_k^\top \nabla f(x_k)\|}\fgrad^\top M_k R_k M_k^\top \fgrad +\frac{\alpha_k^2L_g(1+\varepsilon)}{2n}\| R_k M_k^\top \fgrad\|^2 \\
		&\le g_i(x_k)- \frac23\mu_k\alpha_k \frac{\|M_k^\top \nabla g_i(x_k)\|}{\|M_k^\top \nabla f(x_k)\|}\|Z_k M_k^\top \fgrad\|^2 +\frac{\alpha_k^2L_g(1+\varepsilon)}{n}\|Z_k M_k^\top \fgrad\|^2.
	\end{alignat*}
	The first inequality follows from (\ref{eq:cond_updaterule}) and Lemma~\ref{lemma:Johnson}. The last inequality follows from (\ref{eq:eval_1}) and (\ref{eq:eval_2}).
	Hence, if the step size $\alpha_k$ satisfies 
	\begin{equation}{\label{eq:feasibility direction1 Ik}}
		0\le \alpha_k \le \mu_k\frac{\|M_k^\top \nabla g_i(x_k)\|}{\|M_k^\top \nabla f(x_k)\|} \frac{2n}{3L_g(1+\varepsilon)},
	\end{equation}
	$g_i(x_{k+1})\le 0$ holds for all $i \in \mathcal{A}_k$. We now show that a nonzero lower bound of $\alpha_k$  exists by computing 
        a lower bound for $\frac{\|M_k^\top \nabla g_i(x_k)\|}{\|M_k^\top \nabla f(x_k)\|}$. 
	We apply Lemma~\ref{lemma:Johnson} to $\fgrad$ and $\nabla g_i(x_k)\;(i\in \mathcal{A}_k)$; from Assumptions~\ref{Assumption:Objective} and {\ref{Assumption:minimum of Constraints gradient}},
	it follows that
	\begin{equation}{\label{eq:lower1}}
		\frac{\|M_k^\top \nabla g_i(x_k)\|}{\|M_k^\top \fgrad\|}
		\ge \sqrt{\frac{(1-\varepsilon)}{(1+\varepsilon)}}\frac{\|\nabla g_i(x_k)\|}{\|\fgrad\|}
		\ge \sqrt{\frac{(1-\varepsilon)}{(1+\varepsilon)}} \frac{l_g^{\varepsilon_0}}{U_f}\qquad (i \in \mathcal{A}_k)
	\end{equation}
	holds with probability at least $1-2(\vert\mathcal{A}_k\vert+1)\exp{(-C_0\varepsilon^2d)}$. This inequality shows that $\frac{\|M_k^\top \nabla g_i(x_k)\|}{\|M_k^\top \fgrad\|}$ has a nonzero lower bound. Next we find a lower bound of $\mu_k = \frac{1}{2\sqrt{s_k^\top (G_k^\top M_k M_k^\top G_k)^{-1} s_k}}$. 
	First, we compute an upper bound for $s_k^\top (G_k^\top M_k M_k^\top G_k)^{-1} s_k$. We apply Lemma~\ref{lemma:Johnson} to $\nabla g_i(x_k)\;(i\in \mathcal{A}_k)$ with $M_k = \frac{1}{n} P_k^\top $; from (\ref{eq:eig_linear}), we find that
	\begin{alignat}{2}
		\notag
		s_k^\top (G_k^\top M_k M_k^\top G_k)^{-1} s_k&
		\le \frac{\|s_k\|^2}{\lambda_{\mathrm{min}}(G_k^\top M_k M_k^\top G_k)} &\\
		\notag
		&\le \frac{n^2}{d(1-\varepsilon)\lambda_{\mathrm{min}}(G_k^\top G_k)} \sum_{i \in \mathcal{A}_k}\|M_k^\top \nabla g_i(x_k)\|^2\\
		\notag
		&\le \frac{(1+\varepsilon)}{(1-\varepsilon)\lambda_{\mathrm{min}}(G_k^\top G_k)} \sum_{i \in \mathcal{A}_k}\|\nabla g_i(x_k)\|^2 \\
		\notag
		&< \frac{\vert \mathcal{A}_k\vert U_g^2(1+\varepsilon)}{(1-\varepsilon)\lambda_{\mathrm{min}}(G_k^\top G_k)} \le \frac{dU_g^2(1+\varepsilon)}{(1-\varepsilon)\lambda_{\mathrm{min}}(G_k^\top G_k)}\\
		{\label{eq:lower2}}
		&\le \frac{dU_g^2(1+\varepsilon)}{(1-\varepsilon)\lambda_{\mathrm{min}}^*}
	\end{alignat}
	holds. The second inequality follows from (\ref{eq:eig_linear}) and the third inequality follows from Lemma~\ref{lemma:Johnson}.
	The 4th and 5th inequalities come from Assumptions~\ref{Assumption:existence of inverse}\ref{Assumption:reduced dimension is larger than active sets} and \ref{Assumption:Constraints}.
	Inequality 	\eqref{eq:lower2} implies that $\mu_k$ has a lower bound. Hence, upon combining (\ref{eq:feasibility direction1 Ik}), (\ref{eq:lower1}) and (\ref{eq:lower2}),
	we see that if the step size $\alpha_k$ satisfies
	\begin{equation*}{\label{eq:step size lower bound1 with first direction}}
		0\le \alpha_k \le \frac{\sqrt{\lambda_{\mathrm{min}}^*} l_g^{\varepsilon_0}(1-\varepsilon)}{3U_fU_gL_g(1+\varepsilon)^2}\frac{n}{\sqrt{d}},
	\end{equation*}
	then $g_i(x_k)\le 0 \;(i\in \mathcal{A}_k)$ holds.
	
	As for the nonactive constraints (i.e., $i\notin \mathcal{A}_k$), from the $L_g$-smoothness of constraint functions (\ref{eq:Lg-smooth of constraints}), we have that for all $i \notin \mathcal{A}_k$,
	\[
	g_i(x_{k+1})\le g_i(x_k) + \alpha_k \| \nabla g_i(x_k)\|\| M_kR_kM_k^\top \nabla f(x_k)\| +\frac{\alpha_k^2L_g}{2}\|M_kR_kM_k^\top \nabla f(x_k)\|^2.
	\]
	In solving the quadratic inequality,
	\[
		g_i(x_k) +  \| \nabla g_i(x_k)\|z +\frac{L_g}{2}z^2 \le 0
	\]
	with $z = \alpha_k \| M_kR_kM_k^\top \nabla f(x_k)\|$, we find that
	if the step size $\alpha_k$ satisfies 
	\begin{equation*}
		0\le \alpha_k \le \frac{1}{L_g\|M_kR_kM_k^\top \nabla f(x_k)\|}\frac{-2L_g g_i(x_k)}{\|\nabla g_i(x_k)\| + \sqrt{\|\nabla g_i(x_k)\|^2 - 2L_g g_i(x_k)}},
	\end{equation*}
	then $g_i(x_{k+1})\le 0 \; (i \notin \mathcal{A}_k)$.
	Assumptions~\ref{Assumption:Objective}\ref{Level Set is compact or feasible set is compact}, \ref{Assumption:Constraints}, and \ref{Assumption:maximum of Constraints} yield
	$\|\nabla g_i\|\le U_g,-\infty < g_*\le g_i\le u_g^{\varepsilon_0}<0$. 
	From these relations, we find that $x_{k+1}$ is feasible if the step size $\alpha_k$ satisfies
	\begin{equation}{\label{eq:feasibility direction1 not Ik}}
		0\le \alpha_k \le \frac{1}{\|M_kR_kM_k^\top \nabla f(x_k)\|} \frac{-2u_g^{\varepsilon_0}}{U_g+\sqrt{U_g^2 - 2L_gg_*}}.
	\end{equation}
	From Lemma~\ref{lemma:Johnson} and (\ref{eq:eval_2}),
	\begin{alignat}{2}
		\notag
		\|M_k R_k M_k^\top \fgrad\|^2 &
		\le \frac{1+\varepsilon}{n}\|R_kM_k^\top \fgrad\|^2\\
		\notag
		&\le \frac{2(1+\varepsilon)}{n}\|Z_k M_k^\top \fgrad\|^2\\ 
		{\label{eq:lower3}}
		\notag
		&\le \frac{2(1+\varepsilon)}{n}\|M_k^\top \fgrad\|^2\\
		&\le \frac{2d(1+\varepsilon)^2}{n^3}U_f^2
	\end{alignat}
	hold with probability at least $1-2\exp{(-C_0\varepsilon^2n)}-2\exp{(-C_0\varepsilon^2d)}$. 
	The first inequality follows from Lemma~\ref{lemma:Johnson} with $M_k = \frac{1}{n}P_k^\top$ and the second inequality follows from (\ref{eq:eval_2}).
	The third inequality follows from (\ref{eq:orthogonal projection norm}) and the last inequality follows from Lemma~\ref{lemma:Johnson} with $M_k^\top = \frac{1}{n}P_k$ \; and \;Assumption~\ref{Assumption:Objective}\ref{Level Set is compact or feasible set is compact}.
	(\ref{eq:feasibility direction1 not Ik}) and (\ref{eq:lower3}) together yield 
	\begin{equation*}{\label{eq:step size lower bound2 with first direction}}
		0\le \alpha_k \le \frac{1}{U_f}\sqrt{\frac{n^3}{2d(1+\varepsilon)^2}}\frac{-2u_g^{\varepsilon_0}}{U_g+\sqrt{U_g^2 - 2L_gg_*}}.   
	\end{equation*}
\end{proof}
	The upper bound of the step size $\alpha_k$ in Proposition~\ref{proposition:feasibility in algorithm3 1} consists of two terms, an $O(n/d^{1/2})$ term and an $O(n^{3/2}/d^{1/2})$ term.
	When the original dimension $n$ is large enough, the $O(n^{3/2}/d^{1/2})$ term  becomes larger than the $O(n/d^{1/2})$ term. Accordingly, the step size condition becomes
	\[
		0\le \alpha_k \le \frac{\sqrt{\lambda_{\mathrm{min}}^*} l_g^{\varepsilon_0}(1-\varepsilon)}{3U_fU_gL_g(1+\varepsilon)^2}\frac{n}{\sqrt{d}}.
	\]

  Next, we prove that there exists a non-zero lower bound for the step size in the second direction $d_k = -M_k^\top G_ky_k$. 
\begin{proposition}{\label{proposition:feasibility in algorithm3 2}}
	Suppose that Assumptions~\ref{Assumption:Objective}\ref{Level Set is compact or feasible set is compact},\ref{Assumption:existence of inverse}, \ref{Assumption:Constraints}, and \ref{Assumption:maximum of Constraints} hold and that  $x_k$ is feasible and $\min_i \lambda_i^{(k)} < -\varepsilon_2$. 
	Then, if the step size $\alpha_k$ satisfies
	\begin{align*}
			&0\le 
	\alpha_k \le \min\left(\frac{2(1-\varepsilon)\lambda_{\mathrm{min}}^*}{\varepsilon_2(1+\varepsilon)L_gd},
	\frac{1}{U_fL_g}\sqrt{\frac{(\lambda_{\mathrm{min}}^*)^3(1-\varepsilon)^3}{d^3(1+\varepsilon)^3}},\right.\\
	&\hspace{100pt}\left.
	\sqrt{\frac{(1-\varepsilon)\lambda_{\mathrm{min}}^*}{1+\varepsilon}}\frac{\sqrt{n}}{d\varepsilon_2}\frac{-2u_g^{\varepsilon_0}}{U_g+\sqrt{U_g^2 - 2L_gg_*}}
	\right),
	\end{align*}
	when $d_k = -\frac{\varepsilon_2d}{n} M_k^\top G_k(G_k^\top M_k M_k^\top G_k)^{-1}\bar{d}^{(k)}$ of (\ref{eq:direction2 in RSG-NC}) with $\bar{d}^{(k)}$ defined by (\ref{eq:d-bar1 in direction2 of RSG-NC}) or (\ref{eq:d-bar2 in direction2 of RSG-NC}) in Algorithm~\ref{Algorithm:Third Algorithm}, 
	$x_{k+1}$ is feasible with probability at least $1-2\exp{(-C_0\varepsilon^2n)}-6\exp{(-C_0\varepsilon^2d)}$.
\end{proposition}
\begin{proof}
	Regarding the active constraints, from the $L_g$-smoothness of $g_i$ to (\ref{eq:Lg-smooth of constraints}) and $x_{k+1}- x_{k} = -\alpha_k M_k M_k^\top G_k y_k$ of (\ref{eq:iterate2 in RSG-NC}), we have
	\[
		g_i(x_{k+1}) \le g_i(x_k) -\alpha_k \nabla g_i(x_k)^\top M_kM_k^\top G_ky_k +\frac{L_g}{2}\alpha_k^2 \|M_kM_k^\top G_k y_k \|^2,
	\]
        which leads to
	\begin{alignat}{2}
		\notag
		g_i(x_{k+1}) & \le   g_i(x_k) -\alpha_k \nabla g_i(x_k)^\top M_kM_k^\top G_ky_k +\frac{L_g}{2}\alpha_k^2 \frac{\varepsilon_2^2 d(1+\varepsilon)}{n(1-\varepsilon)\lambda_{\mathrm{min}}(G_k^\top G_k)}\|\bar{d}^{(k)}\|^2 \\
		\notag
		& = g_i(x_k) - \alpha_k \frac{\varepsilon_2d}{n}\bar{d}_i^{(k)} +\frac{L_g}{2}\alpha_k^2 \frac{\varepsilon_2^2 d(1+\varepsilon)}{n(1-\varepsilon)\lambda_{\mathrm{min}}(G_k^\top G_k)}\|\bar{d}^{(k)}\|^2 \\
		\notag
		& = g_i(x_k) - \alpha_k \frac{\varepsilon_2d}{n}\bar{d}_i^{(k)} +\frac{L_g}{2}\alpha_k^2 \frac{\varepsilon_2^2 d(1+\varepsilon)}{n(1-\varepsilon)\lambda_{\mathrm{min}}^*}\|\bar{d}^{(k)}\|^2,
	\end{alignat}
with probability at least $1-2\exp{(-C_0\varepsilon^2n)}-2\exp{(-C_0\varepsilon^2d)}$.
The first inequality follows from (\ref{eq:norm of MMGy}) and the first equality follows from definition (\ref{def:yk}) of $y_k$.
Hence, if the step size satisfies $-\alpha_k \bar{d}^{(k)}_i +\frac{L_g}{2}\alpha_k^2 \frac{\varepsilon_2(1+\varepsilon)}{(1-\varepsilon)\lambda_{\mathrm{min}}^*}\|\bar{d}^{(k)}\|^2\le 0$, this direction preserves feasibility.
Therefore, we have
\begin{equation}{\label{eq:step size with second direction}}
	0\le \alpha_k \le \frac{2(1-\varepsilon)\lambda_{\mathrm{min}}^*}{\varepsilon_2(1+\varepsilon)L_g}\frac{\bar{d}_i^{(k)}}{\|\bar{d}^{(k)}\|^2}.
\end{equation}

Next, we evaluate the lower bound of $\frac{\bar{d}_i^{(k)}}{\|\bar{d}^{(k)}\|^2}$ with $\bar{d}^{(k)}$ defined by (\ref{eq:d-bar1 in direction2 of RSG-NC}) and by (\ref{eq:d-bar2 in direction2 of RSG-NC}). From the proof of Proposition~\ref{Proposition:dec_obj_extalgo_2}, $\|\bar{d}^{(k)}\|^2 \le \vert \mathcal{A}_k \vert$ holds for both (\ref{eq:d-bar1 in direction2 of RSG-NC}) and (\ref{eq:d-bar2 in direction2 of RSG-NC}).

In the case of $\bar{d}^{(k)} = \bm{1}$, i.e., (\ref{eq:d-bar1 in direction2 of RSG-NC}) in Algorithm~\ref{Algorithm:Third Algorithm}, we have that $\bar{d}_i^{(k)}$ = 1 and $\bar{d}^{(k)}_i/\|\bar{d}^{(k)}\|^2 \ge  1/\vert\mathcal{A}_k\vert  > 1/d$ from Assumption~\ref{Assumption:existence of inverse}\ref{Assumption:reduced dimension is larger than active sets}. 

In the case of $\bar{d}^{(k)}_i = \chi_{-}(\lambda_i^{(k)}) + \frac{\sum_{\lambda_j^{(k)} \le 0 }-\lambda_j^{(k)}}{2\sum_{\lambda_j^{(k)}>0} \lambda_j^{(k)}}\chi_{+}(\lambda_i^{(k)})$, i.e., (\ref{eq:d-bar2 in direction2 of RSG-NC}) in Algorithm~\ref{Algorithm:Third Algorithm}, when $\lambda_{i}^{(k)}$ is non-positive, $\bar{d}_i^{(k)} = 1$ and 
$\bar{d}^{(k)}_i/\|\bar{d}^{(k)}\|^2  > 1/d$.
When $\lambda_i^{(k)}> 0$, we have
\begin{alignat}{2}{\label{eq:lower d/d^2}}
	\notag
	\frac{\bar{d}^{(k)}_i}{\|\bar{d}^{(k)}\|^2} &
	\ge \frac{\sum_{\lambda^{(k)}_j \le 0} (-\lambda^{(k)}_j)}{2\vert\mathcal{A}_k\vert\sum_{\lambda^{(k)}_j > 0}\lambda^{(k)}_j}\\
	\notag
	&> \frac{\varepsilon_2}{2\vert \mathcal{A}_k\vert \|\lambda^{(k)}\|_1}\\
	&> \frac{\varepsilon_2}{2d\|\lambda^{(k)}\|_1} > 0.
\end{alignat}
The first inequality follows from the definition of $\bar{d}_i^{(k)}$ of (\ref{eq:d-bar2 in direction2 of RSG-NC}) and $\|\bar{d}^{(k)}\|\le \sqrt{\vert\mathcal{A}_k\vert}$.
The second inequality follows from $\min_i\lambda^{(k)}_i < -\varepsilon_2$ and the last inequality follows from Assumption~\ref{Assumption:existence of inverse}\ref{Assumption:reduced dimension is larger than active sets}.
Furthermore, from (\ref{eq:upper bound of lambda}), we have
\begin{equation}{\label{eq:l1 norm of lambda}}
	\|\lambda^{(k)}\|_1 \le \sqrt{\vert\mathcal{A}_k\vert}\|\lambda^{(k)}\| \le \sqrt{\frac{d(1+\varepsilon)}{(1-\varepsilon)}} \frac{U_f}{\sqrt{\lambda_{\mathrm{min}}(G_k^\top G_k)}}\le \sqrt{\frac{d(1+\varepsilon)}{(1-\varepsilon)}} \frac{U_f}{\sqrt{\lambda_{\mathrm{min}}^*}}
\end{equation}
with probability at least $1-4\exp{(-C_0\varepsilon^2d)}$. 
Combining (\ref{eq:lower d/d^2}) and (\ref{eq:l1 norm of lambda}), we have
\begin{equation}{\label{eq:lower d/d^2 2}}
	\frac{\bar{d}^{(k)}_i}{\|\bar{d}^{(k)}\|^2} \ge \frac{1}{2U_f}\sqrt{\frac{\lambda_{\mathrm{min}}^*(1-\varepsilon)}{d^3(1+\varepsilon)}}\varepsilon_2.
\end{equation}
From (\ref{eq:step size with second direction}) and, (\ref{eq:lower d/d^2 2}) or $\bar{d}^{(k)}_i/\|\bar{d}^{(k)}\|^2 > 1/d$, if the step size $\alpha_k$ satisfies
\begin{equation*}
	0\le
	\alpha_k \le \min\left(\frac{2(1-\varepsilon)\lambda_{\mathrm{min}}^*}{\varepsilon_2(1+\varepsilon)L_gd},
	\frac{1}{U_fL_g}\sqrt{\frac{(\lambda_{\mathrm{min}}^*)^3(1-\varepsilon)^3}{d^3(1+\varepsilon)^3}}\right),
\end{equation*}
$g_i(x_{k+1})\le 0 $ is satisfied for the active constraints.

If $i\notin \mathcal{A}_k $, we can apply the same argument as in (\ref{eq:feasibility direction1 not Ik}) of Proposition~{\ref{proposition:feasibility in algorithm3 1}} by replacing $\|M_k^\top R_k M_k^\top \fgrad\|$ with $\|M_kM_k^\top G_k y_k \|$. Thus, we have
\begin{equation}{\label{eq:lower bound in algorithm3}}
	0\le \alpha_k \le \frac{1}{\|M_kM_k^\top G_k y_k \|} \frac{-2u_g^{\varepsilon_0}}{U_g+\sqrt{U_g^2 - 2L_gg_*}}.
\end{equation}
From (\ref{eq:norm of MMGy}), (\ref{eq:lower bound in algorithm3}), and $\|\bar{d}^{(k)}\|\le \sqrt{\vert\mathcal{A}_k\vert}\le \sqrt{d}$, if the step size $\alpha_k$ satisfies
\[
	0\le
        \alpha_k \le \sqrt{\frac{(1-\varepsilon)\lambda_{\mathrm{min}}^*}{1+\varepsilon}}\frac{\sqrt{n}}{d\varepsilon_2}\frac{-2u_g^{\varepsilon_0}}{U_g+\sqrt{U_g^2 - 2L_gg_*}},
\]
$g_i(x_{k+1}) \le 0$ for the non-active constraints with probability at least $1-2\exp{(-C_0\varepsilon^2n)}-2\exp{(-C_0\varepsilon^2d)}$.
\end{proof}
	Following a similar argument to Proposition~\ref{proposition:feasibility in algorithm3 1}, the upper bound of the step size $\alpha_k$ in Proposition~\ref{proposition:feasibility in algorithm3 2} consists of three terms, an $O(1/d)$ term, an $O(1/d^{3/2})$ term, and an $O(n^{1/2}/d)$ term.
	When the original dimension $n$ is large enough, the $O(n^{1/2}/d)$ term becomes larger than other terms and the step size conditions can be written as 
	\[
		0\le \alpha_k \le \min\left(\frac{2(1-\varepsilon)\lambda_{\mathrm{min}}^*}{\varepsilon_2(1+\varepsilon)L_gd},
		\frac{1}{U_fL_g}\sqrt{\frac{(\lambda_{\mathrm{min}}^*)^3(1-\varepsilon)^3}{d^3(1+\varepsilon)^3}}\right).
	\]
        
\subsection{Global convergence}
We will construct $\eta^{(k)} \in \mathbb{R}^m$ from $\lambda^{(k)}$ of Algorithm~\ref{Algorithm:Third Algorithm} in the same way as described in Section~\ref{sec:global_conv}.
\begin{theorem}
	\label{theorem:global convergence of RSG-NC}
	Suppose that Assumptions~\ref{Assumption:Objective},\ref{Assumption:existence of inverse},\ref{Assumption:Constraints} and \ref{Assumption:maximum of Constraints} hold. 
	Let the optimal value of (\ref{main problem}) be $f^*(> -\infty)$, and let
	\begin{align*}
	\delta(\varepsilon,\varepsilon_0,\varepsilon_1,\varepsilon_2) =& \min\left(
	\min\left(O(n),O\left(\frac{n}{\sqrt{d}}\right),O\left(\sqrt{\frac{n^3}{d}}\vert u_g^{\varepsilon_0}\vert\right)\right)\frac{d}{n^2}\varepsilon_1^2,\right.\\
	&\qquad\left.\min\left(O(d^{-1}),O(\varepsilon_2^{-1}d^{-1}),O(d^{-3/2}),O\left(\frac{\sqrt{n}}{d\varepsilon_2}\vert u_g^{\varepsilon_0}\vert \right)\right) \frac{d}{n}\varepsilon_2^2
	\right).
	\end{align*}
	Then, Algorithm~\ref{Algorithm:Third Algorithm} generates an ($\varepsilon_1,\varepsilon_2,O(\varepsilon_0)$)-KKT pair from inputs $d, \varepsilon_0, \varepsilon_2, \delta_1 = \sqrt{\frac{d(1-\varepsilon)}{n^2}}\varepsilon_1, \mu_k = \frac{1}{2\sqrt{s_k^\top (G_k^\top M_k M_k^\top G_k)^{-1}s_k}}$
	within $K:= \left\lceil \frac{f(x_0)- f^*}{\delta(\varepsilon,\varepsilon_0,\varepsilon_1,\varepsilon_2)}\right\rceil$ iterations with probability at least $1 - 4K\exp{(-C_0\varepsilon^2n)}-2K(d+6)\exp{(-C_0\varepsilon^2d)}$.
\end{theorem}
\begin{proof}
	The points $\{x_k\}$ are feasible because the conditions of Propositions~\ref{proposition:feasibility in algorithm3 1} and \ref{proposition:feasibility in algorithm3 2} are satisfied.
	Hence, (\ref{definite:feasible}) is satisfied. Furthermore, we can prove (\ref{definite:epsilon kkt 3}) in a similar way to (\ref{eq:eval_cond3}) in Theorem~\ref{theorem:global convergence}.
	Next, we prove (\ref{definite:epsilon kkt 1}) and (\ref{definite:epsilon kkt 2}). If Algorithm~\ref{Algorithm:Third Algorithm} stops, we have 
	\begin{equation}{\label{eq:first norm}}
		\|d_k\|\le \delta_1
	\end{equation}
	and 
	\[
	\min_{i} \lambda_i^{(k)} \geq -\varepsilon_2.
	\]
	The second inequality is identical to (\ref{definite:epsilon kkt 2}).
	From (\ref{eq:first norm}) and \eqref{eq:direction1 in RSG-NC}, 
  we obtain
	\begin{alignat*}{2}
		\notag
		\delta_1 &= \sqrt{\frac{d(1-\varepsilon)}{n^2}}\varepsilon_1\ge \|M_k^\top (\fgrad + G_k \bar{\lambda}^{(k)})\| = \|R_k M_k^\top \fgrad\|\\
		\notag
		&\ge \|Z_k M_k^\top \fgrad \| = \|M_k^\top (\fgrad + G_k \lambda^{(k)})\|\\
		&\ge \sqrt{\frac{d(1-\varepsilon)}{n^2}}\|\fgrad + G_k \lambda^{(k)}\|.
	\end{alignat*}
	The second inequality follows from (\ref{eq:eval_2}) and the last inequality follows from Lemma~\ref{lemma:Johnson}.
	Then, 
	\[
		\varepsilon_1 \ge \|\fgrad + G_k \lambda^{(k)}\| = \|\fgrad + \sum_{i = 1}^m \eta_{i}^{(k)} \nabla g_i(x_k)\|
	\]
	holds with probability at least $1-2\exp{(-C_0\varepsilon^2d)}$, and we have confirmed that
	$(x_k,\eta^{(k)})$ is an $\left(\varepsilon_1,\varepsilon_2, O(\varepsilon_0)\right)$-KKT pair.
        
	Now let us prove that Algorithm~\ref{Algorithm:Third Algorithm} terminates at the $\bar{k}$th iteration with $\bar{k} \leq K$,
        by using the same argument as in Theorem~\ref{theorem:global convergence}.
        Assuming an arbitrary iteration $k \le \bar{k}-1$, we will show that the function value strictly and monotonically decreases in the two directions.
	\begin{description}
		\item{Case 1:} 
		When $\|M_k^\top(\fgrad + G_k\bar{\lambda}^{(k)})\|> \delta_1$ and $0 \le \alpha_k\le \frac{n}{3L(1+\varepsilon)}$ hold, 
		we have
		\[
		\delta_1^2 = \frac{d(1-\varepsilon)}{n^2}\varepsilon_1^2 
		<\|M_k^\top(\fgrad + G_k\bar{\lambda}^{(k)})\|^2
		=\|R_kM_k^\top \fgrad\|^2
		< 2\|Z_k M_k^\top \fgrad\|^2
		\]
		from (\ref{eq:eval_2}) and 
		\[
			\alpha_k - \frac{3\alpha_k^2 L(1+\varepsilon)}{2n} \ge \frac{1}{2} \alpha_k.
		\]
		These relations together with Proposition~\ref{Proposition:dec_obj_extalgo} lead us to
		\begin{equation*}
			f(x_{k+1}) - f(x_k) \le - \frac16 \alpha_k \frac{d(1-\varepsilon)}{n^2}\varepsilon_1^2,
		\end{equation*}
		when the step size $\alpha_k$ satisfies $0 \le \alpha_k \le \frac{n}{3L(1+\varepsilon)}$.
		For the first direction $d_k = - R_kM_k^\top \fgrad$ of (\ref{eq:direction1 in RSG-NC}), Proposition~\ref{proposition:feasibility in algorithm3 1} allows us to set the step size $\alpha_k$ as
		\begin{alignat*}{2}         
			\alpha_k & = \min\left(
			\frac{n}{3L(1+\varepsilon)},
			\frac{\sqrt{\lambda_{\mathrm{min}}^*} l_g^{\varepsilon_0}(1-\varepsilon)}{3U_fU_gL_g(1+\varepsilon)^2}\frac{n}{\sqrt{d}},
			\frac{1}{U_f}\sqrt{\frac{n^3}{2d(1+\varepsilon)^2}}\frac{-2u_g^{\varepsilon_0}}{U_g+\sqrt{U_g^2 - 2L_gg_*}}
			\right)\\
			&= \min\left(O(n),O\left(\frac{n}{\sqrt{d}}\right),O\left(\sqrt{\frac{n^3}{d}}\vert u_g^{\varepsilon_0}\vert \right)\right).
		\end{alignat*}
		Then,
		\begin{equation}{\label{eq:decrease1}}
			f(x_{k+1}) - f(x_k) \le - \min\left(O(n),O\left(\frac{n}{\sqrt{d}}\right),O\left(\sqrt{\frac{n^3}{d}}\vert u_g^{\varepsilon_0}\vert \right)\right)\frac{d}{n^2}\varepsilon_1^2.    
		\end{equation}
		
		\item{Case 2:} 
		When $\|M_k^\top \fgrad + M_k^\top G_k \bar{\lambda}^{(k)}\| \le \delta_1$, we update the point by $d_k = M_k^\top G_k y_k$. Since Algorithm~\ref{Algorithm:Third Algorithm} does not terminate at iteration $k$, we have $\min_i \lambda_i^{(k)} < -\varepsilon_2$.
		When the step size $\alpha_k$ satisfies $0 \le \alpha_k \le \frac{(1-\varepsilon)\lambda_{\min}(G_k^\top G_k)}{2(1+\varepsilon)\vert \mathcal{A}_k\vert L}$, from Proposition~\ref{Proposition:dec_obj_extalgo_2} and the following inequality,
		\[
			\alpha_k - \frac{(1+\varepsilon)L \alpha_k^2 \vert \mathcal{A}_k \vert}{(1-\varepsilon)\lambda_{\min}(G_k^\top G_k)} \ge \frac{1}{2}\alpha_k,
		\]
		we have
		\begin{equation*}
				f(x_{k+1}) - f(x_{k}) \le -\alpha_k \frac{\varepsilon_2^2 d}{4n}.
		\end{equation*}
		By Proposition~\ref{proposition:feasibility in algorithm3 2}, we can set the step size $\alpha_k$ to
		{\small 
		\begin{align*}        
			\alpha_k & = \min\left(
			\frac{(1-\varepsilon) \lambda_{\mathrm{\mathrm{min}}}^*}{2(1+\varepsilon) Ld},
			\frac{2(1-\varepsilon)\lambda_{\mathrm{min}}^*}{\varepsilon_2(1+\varepsilon)L_gd},
			\frac{1}{U_fL_g}\sqrt{\frac{(\lambda_{\mathrm{min}}^*)^3(1-\varepsilon)^3}{d^3(1+\varepsilon)^3}},\right.\\
			&\quad \left.\sqrt{\frac{(1-\varepsilon)\lambda_{\mathrm{min}}^*}{1+\varepsilon}}\frac{\sqrt{n}}{d\varepsilon_2}\frac{-2u_g^{\varepsilon_0}}{U_g+\sqrt{U_g^2 - 2L_gg_*}}
			\right)\\
			&= \min\left(O(d^{-1}),O(\varepsilon_2^{-1}d^{-1}),O(d^{-3/2}),O\left(\frac{\sqrt{n}}{d\varepsilon_2}\vert u_g^{\varepsilon_0}\vert \right)\right).
		\end{align*}}
	\normalsize
	Accordingly, we have
		\begin{equation}{\label{eq:decrease2}}
			f(x_{k+1}) - f(x_k)\le -\min\left(O(d^{-1}),O(\varepsilon_2^{-1}d^{-1}),O(d^{-3/2}),O\left(\frac{\sqrt{n}}{d\varepsilon_2}\vert u_g^{\varepsilon_0}\vert \right)\right) \frac{d}{n}\varepsilon_2^2.
		\end{equation}
		
	\end{description}
	From the relations (\ref{eq:decrease1}) and (\ref{eq:decrease2}), Algorithm~\ref{Algorithm:Third Algorithm} decreases the objective function value by 
	$\delta(\varepsilon,\varepsilon_0,\varepsilon_1,\varepsilon_2)$: 
	\[
	f(x_{k+1}) - f(x_k)\le- \delta(\varepsilon,\varepsilon_0,\varepsilon_1,\varepsilon_2) <0.
	\]
        Summing over $k$, we find that
	\[
        f^* - f(x_0) \leq f(x_{\bar{k}}) - f(x_0) \le -\bar{k} \delta(\varepsilon,\varepsilon_0,\varepsilon_1,\varepsilon_2), 
	\]
        which implies $\bar{k} \leq K$,
	with probability at least $1 - 4\bar{k}\exp{(-C_0\varepsilon^2n)}-2\bar{k}(d+6)\exp{(-C_0\varepsilon^2d)}$.
\end{proof}
\begin{remark}
  If the original dimension $n$ is large enough, the denominator of the iteration number $K$,  $\delta$,
 becomes
	\begin{align*}
	\delta(\varepsilon,\varepsilon_0,\varepsilon_1,\varepsilon_2) =& \min\left(
	\min\left(O(n),O\left(\frac{n}{\sqrt{d}}\right),\right)\frac{d}{n^2}\varepsilon_1^2,\right.\\
	&\qquad\left.\min\left(O(d^{-1}),O(\varepsilon_2^{-1}d^{-1}),O(d^{-3/2}) \right) \frac{d}{n}\varepsilon_2^2
	\right),
	\end{align*}
	and we can ignore the terms of $\vert u_g^{\varepsilon_0}\vert $.
\end{remark}
\begin{remark}
	We can prove convergence of the deterministic of our algorithm (i.e., $M_k = I$) by the same argument in Section~\ref{section:Algorithm for nonlinear inequality constraints}.
	However, the iteration complexity becomes $O(\max{ (\max{(\varepsilon_1^{-2}, \varepsilon_2^{-2})}, (\vert u_g^{\varepsilon_0} \vert)^{-1}\max{(\varepsilon_1^{-2}, \varepsilon_2^{-2})} ) })$ and we cannot ignore $u_g^{\varepsilon_0}$ terms.
	When calculating the gradient $\nabla f(x_k)$ is difficult, the time complexity of the deterministic version to reach an approximate KKT point becomes $O(n) \times O(\max{ (\max{(\varepsilon_1^{-2}, \varepsilon_2^{-2})}, (\vert u_g^{\varepsilon_0} \vert)^{-1}\max{(\varepsilon_1^{-2}, \varepsilon_2^{-2})} ) })$, which is worse than ours with randomness ($O(d) \times O(\max(\frac{n}{\sqrt{d}}\varepsilon_1^{-2},n\sqrt{d}\varepsilon_2^{-2}))$).
\end{remark}
	The computational complexity per iteration of the proposed method is 
	\[
	O\left(dn\vert\mathcal{A}_k\vert  + m(T_{grad} + T_{value}) + mT_{value}N_{\mathrm{while}}\right).
	\] 
	$N_{\mathrm{while}}$ denotes the number of executions of the while-loop to satisfy feasibility and $N_{\mathrm{while}}$ is  $O\left(\vert \log{\frac{\sqrt{d}}{n}} \vert \right)$ or $O(\log{d})$.
The first and second terms come from calculating $M_k^\top G_k$ and the active set, respectively. The last term comes from the while-loop.
From Propositions~\ref{proposition:feasibility in algorithm3 1} and \ref{proposition:feasibility in algorithm3 2}, if 
\[
h\beta^k \le \min\left(O\left(\frac{n}{\sqrt{d}}\right),O\left(\sqrt{\frac{n^3}{d}}\vert u_g^{\varepsilon_0}\vert \right)\right)
\]
with the direction of (\ref{eq:direction1 in RSG-NC})
or
\[
	h\beta^k \le \min\left( O(\varepsilon_0^{-1}d^{-1}),
	O(d^{-3}),
	O\left(\frac{\sqrt{n}\vert u_g^{\varepsilon_0}\vert }{d\varepsilon_2}\right)\right)
\]
with the direction of (\ref{eq:direction2 in RSG-NC})
is satisfied, the while-loop will terminate. Hence, we can find a feasible solution within
at least $O\left(\vert \log{\frac{\sqrt{d}}{n}} \vert\right)$ or $O(\log{d})$ steps.

\section{Numerical experiments}
In this section, we provide results for test problems and machine-learning problems using synthetic and real-world data.
For comparison, we selected projected gradient descent (PGD) and the gradient projection method (GPM)~\cite{rosen1960gradient}.  Furthermore, 
we compared our method with a deterministic version constructed by setting $M_k$ to an identity matrix,
which is the same as in GPM when setting $\varepsilon_0=0$ of $\mathcal{A}_k$ for all $k$.  
All programs were coded in python 3.8  and run on a machine with Intel(R) Xeon(R) CPU E5-2695 v4 @ 2.10GHz and Nvidia (R) Tesla (R) V100 SXM2 16GB.

\subsection{Linear constrained problems}
\subsubsection{Nonconvex quadratic objective function}
We applied Algorithm~\ref{Secondary Algorithm} to a nonconvex quadratic function under box constraints:
\begin{equation*}
	\begin{split}
		&\min_{x} \frac{1}{2}x^\top Q x + b^\top x,\\
		&\mathrm{s.t.} \; -1\le x \le 1.
	\end{split}
\end{equation*}
We used $Q\in \mathbb{R}^{1000\times 1000},b\in \mathbb{R}^{1000}$, whose entries were sampled from $\Normal$, and thus, $Q$ was not a positive semi-definite matrix.
We set the parameters $\varepsilon_0, \delta_1, \varepsilon_2, \beta$, and the reduced dimension $d$ as follows:
\[
	\varepsilon_0 = 10^{-6}, \delta_1 = 10^{-4}, \varepsilon_2 = 10^{-6}, \beta = 0.8, d= 1000.
\]
For $M_k = \frac{1}{n}P_k^\top$, we set the step size as  $h \in \{10^{2}n/L,10^{1}n/L,n/L,10^{-1}/L\}$. For $M_k = I$, we set the step sizes as $h \in \{10^{2}/L,10^{1}/L,1/L,10^{-1}/L\}$.
$L$ denotes the maximum eigenvalue of $Q$.
As shown in Table~\ref{table:Numerical Experiment1}, our method with random projections worked better than the deterministic version.
This is because our method could move randomly when there are many stationary points. This result indicates that our randomized subspace algorithm tends to explore a wider space than the deterministic version.
PGD converges to a stationary point faster than our method does, while our method obtains better solutions achieving smaller function values than those of the PGD in most cases for some choices of random matrices.

\begin{table}[hbtp]
	\centering
	\caption{Applying methods starting from $0$ as the initial solution for minimizing a quadratic function under box constraints. The results for Ours ($M_k = \frac{1}{n}P_k^\top$) are the average and the standard deviation among 10 trials with different choices of random matrices.  }
	{\label{table:Numerical Experiment1}}
	\begin{tabular}{ccccc}
		&PGD & GPM &Ours $(M_k=I)$ &  Ours ($M_k = \frac{1}{n}P_k^\top$)  \\
		\hline \hline
		$f(x_k)$ & -21545& -20620& -22172&-21751$\pm$ 400\\
		Time[s] &0.74 & 1198 &146 &172 $\pm$ 25 \\
		\hline
	\end{tabular}
\end{table}
\subsubsection{Non-negative matrix completion}
We applied Algorithm~\ref{Secondary Algorithm} to optimization problems with realistic data.
We used the MovieLens 100k dataset and solved the non-negative matrix completion problem:
\begin{equation*}
	\begin{split}
		&\min_{U,V} \|\mathcal{P}_{\Omega}(X) - \mathcal{P}_{\Omega}(UV^\top) \|^2,\\
		&\mathrm{s.t.} \; U\ge 0, V\ge 0.
	\end{split}
\end{equation*}
Here, $X\in R^{943\times 1682}$ is a data matrix, and $U \in \mathbb{R}^{943\times 5}$ and $V\in \mathbb{R}^{1682\times 5}$ are decision variables.
We set $\Omega \subset \{(i,j) \vert \ i=1,\ldots,943, j=1\ldots,1682\}$
and defined $\mathcal{P}_{\Omega}$ as 
\[
(\mathcal{P}_\Omega(X))_{ij} = 
\left\{\begin{array}{cc}
	X_{ij}&((i,j)\in \Omega), \\
	0&(\mathrm{otherwise}).
\end{array}
\right. 
\]
We set the parameters $\varepsilon_0, \delta_1, \varepsilon_2, \beta$, and reduced dimension $d$ as follows:
\[
	\varepsilon_0 = 10^{-4}, \delta_1 = 10^{-5}, \varepsilon_2 = 10^{-5}, \beta = 0.8, d = 600.
\]
For $M_k = \frac{1}{n}P_k^\top$, we set the step size as $h \in \{10^2n,10^1n,n,10^{-1}n,10^{-2}n\}$. For $M_k = I$, we set the step size as $h \in \{10^2,10^1,1,10^{-1},10^{-2}\}$.
As shown in Table~\ref{table:Numerical Experiment2}, our method with randomness obtained the best result.
It converged to a different point from the point found by the other methods; that made its computation time longer.

\begin{table}[hbtp]
	\centering
	\caption{Applying methods starting from all $1$ as the initial solution for the non-negative matrix completion problem. The results in Ours ($M_k = \frac{1}{n}P_k^\top$) are the average and the standard deviation among 20 trials with different choices of random matrices.   }
	{\label{table:Numerical Experiment2}}
	\begin{tabular}{ccccc}
		&PGD & GPM &Ours $(M_k=I)$ &  Ours ($M_k = \frac{1}{n}P_k^\top$)  \\
		\hline \hline
		$f(x_k)$ & 66580& 66580 &66582&51585 $\pm$ 348\\
		Time[s] &220 & 325 &379 &579 $\pm$ 5\\
		\hline
	\end{tabular}
\end{table}

\subsection{Nonlinear constrained problems}
\subsubsection{Neural network with constraints}
We applied Algorithm~\ref{Algorithm:Third Algorithm} to a high-dimensional optimization having nonlinear constraint(s) as a regularizer $\mathcal{R}$ of \eqref{reg_const_problem}. 
We used three different three-layer neural networks with cross entropy loss functions $\mathcal{L} (x) := \sum_{i=1}^{48000} \ell_i(x)$
for the MNIST dataset, in which the dimension of $x$ is $669706$;  
\begin{description}
\item{(a)} neural network with $l_1$-regularizer $\|x\|_1 \leq 12000$ and sigmoid activation function, 
\item{(b)} neural network with the same $l_1$-regularizer with (a) and ReLu activation function, and
\item{(c)} neural network with fused lasso~\cite{tibshirani2005sparsity}
  $\|x\|_1 \leq 12000, \sum_{i=2}^{669706} \vert x_i-x_{i-1} \vert \leq 14600$ and sigmoid activation function.
\end{description}
We set the parameters $\varepsilon_0, \delta_1, \varepsilon_2, \beta, \mu_k$ as follows:
\begin{align*}
	\varepsilon_0 = 10^{-1}, \delta_1 = 10^{-8}, \varepsilon_2 = 10^{-5}, \beta = 0.8, \\
	\mu_k = \frac{r}{\sqrt{s_k^\top (G_k^\top M_k M_k^\top G_k)s_k}}\; (r \in \{0.5, 0.1, 0.05\}).
\end{align*}

For $M_k = \frac{1}{n}P_k^\top$, we set the step size as $h \in \{10^4n,10^3n,10^2n,10^1n,n,10^{-1}n\}$. For $M_k = I$, we set the step size as $h \in \{10^4.10^3,10^2,10,1,10^{-1}\}$.
We also used the dynamic barrier method~\cite{gong2021automatic} for the $l_1$ regularizer and fused lasso problems, and PGD for the
$l_1$ regularizer problems. 

PGD performed well when the projection onto the constraints could be calculated easily, while the dynamic barrier methods worked well when the number of constraints was equal to one.
Therefore, problem settings (a) and (b) are good for these methods.
Figure~\ref{figure:Numerical Experiment3}(a,b) shows that our method with randomness worked as well as the compared methods under $l_1$-regularization.
Furthermore, it performed better than the deterministic versions, although
we did not prove convergence in the non-smooth-constraints setting due to the $l_1$-norm.
As shown in Figure~\ref{figure:Numerical Experiment3}(c) for the non-simple projection setting, our method with randomness outperformed the compared methods. 
The step size of RSG-NC with $M_k = I$ came close to 0 in order to satisfy feasibility under $l_1$ regularization. On the other hand, our method with $M_k = \frac{1}{n}P_k^\top$ performed well and the step size did not come close to 0.
\begin{figure}[H]
	\begin{minipage}{0.45\hsize}
		\centering
		\includegraphics[scale = 0.4]{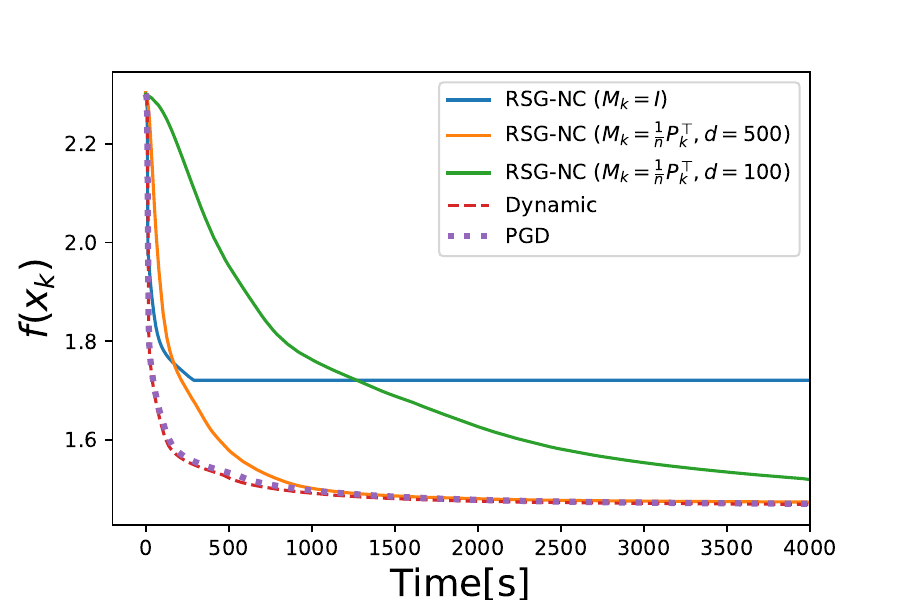}
		\subcaption{}
	\end{minipage}
	\begin{minipage}{0.45\hsize}
		\centering
		\includegraphics[scale = 0.4]{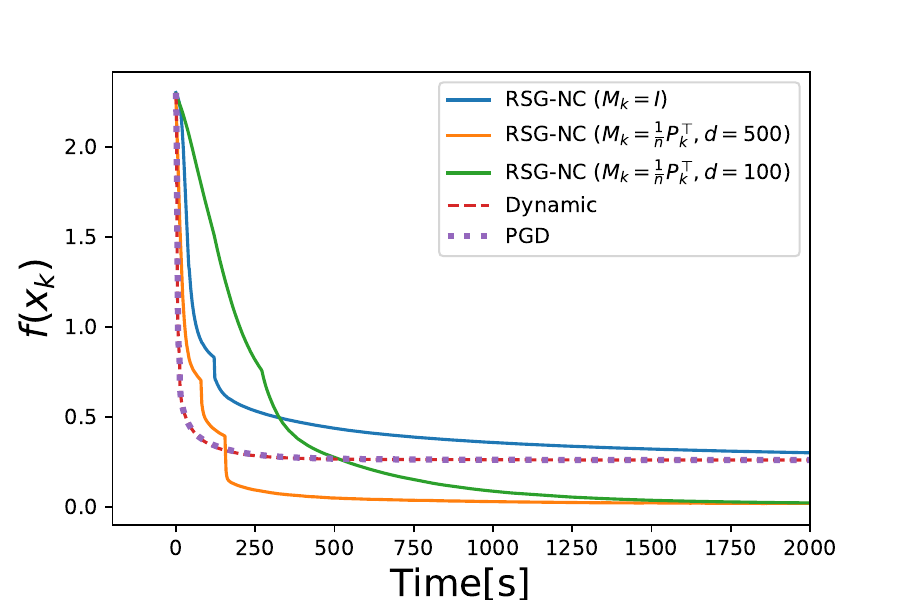}
		\subcaption{}
	\end{minipage}
	\begin{minipage}{\hsize}
		\centering
		\includegraphics[scale = 0.4]{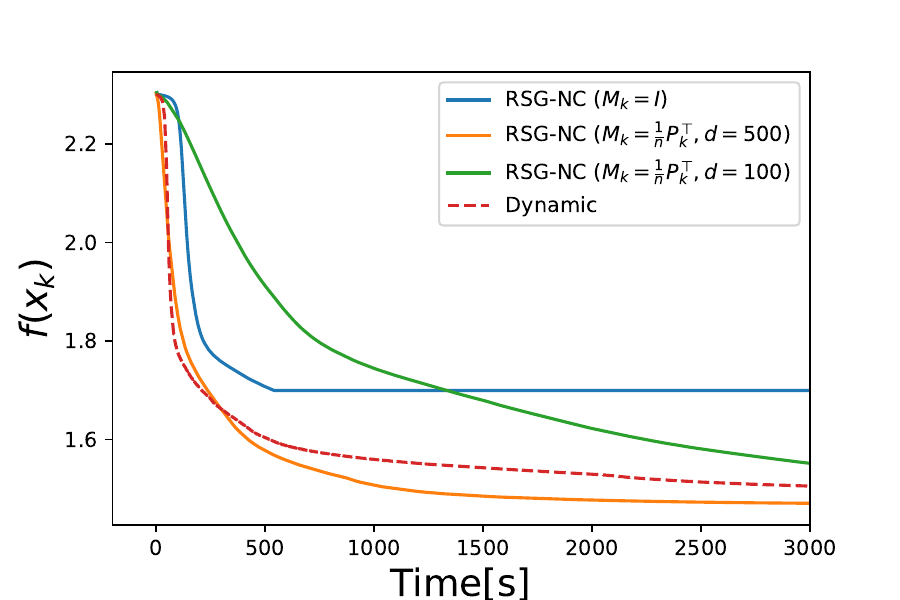}
		\subcaption{}
	\end{minipage}
	\caption{(a) Neural network with $l_1$-regularizer and sigmoid activation function, 
	(b) Neural network with $l_1$-regularizer and ReLu activation function
	(c) Neural network with fused lasso and sigmoid activation function.}
	{\label{figure:Numerical Experiment3}}
\end{figure}

\subsubsection{CNN with constraints}
Since machine-learning problems often have highly nonconvex complicated objective functions consisting of numerous terms,
 we can not obtain the full gradient because of limitations on memory.  
In such a situation, we can use finite difference to calculate gradients. Here, we applied our method to this optimization problem that  
needs only the projected gradients $M_k^\top \nabla f \in \mathbb{R}^{d}$. 
If the reduced dimension $d$ is much smaller than original dimension $n$, it would save on time complexity. We optimized the CNN with the cross entropy loss under $l_2$ regularization, 
$\|x\|_2^2\le 50$, on the MNIST dataset.
We set the parameters $\varepsilon_0, \delta_1, \varepsilon_2, \beta, \mu_k$ as follows:
\begin{align*}
	\varepsilon_0 = 10^{-6}, \delta_1 = 10^{-8}, \varepsilon_2 = 10^{-4}, \beta = 0.8, \mu_k = \frac{1}{2\sqrt{s_k^\top (G_k^\top M_k M_k^\top G_k)s_k}}.
\end{align*}
For $M_k = \frac{1}{n}P_k^\top$, we set the step size as $h \in \{10^2n,10^1n,n\}$. For $M_k = I$, we set the step size as $h \in \{100,10,1\}$.
Figure~\ref{figure:Numerical Experiment4} shows that our method with random projection performs better than the deterministic version. 

\begin{figure}[H]
	\centering
	\includegraphics[scale = 0.5]{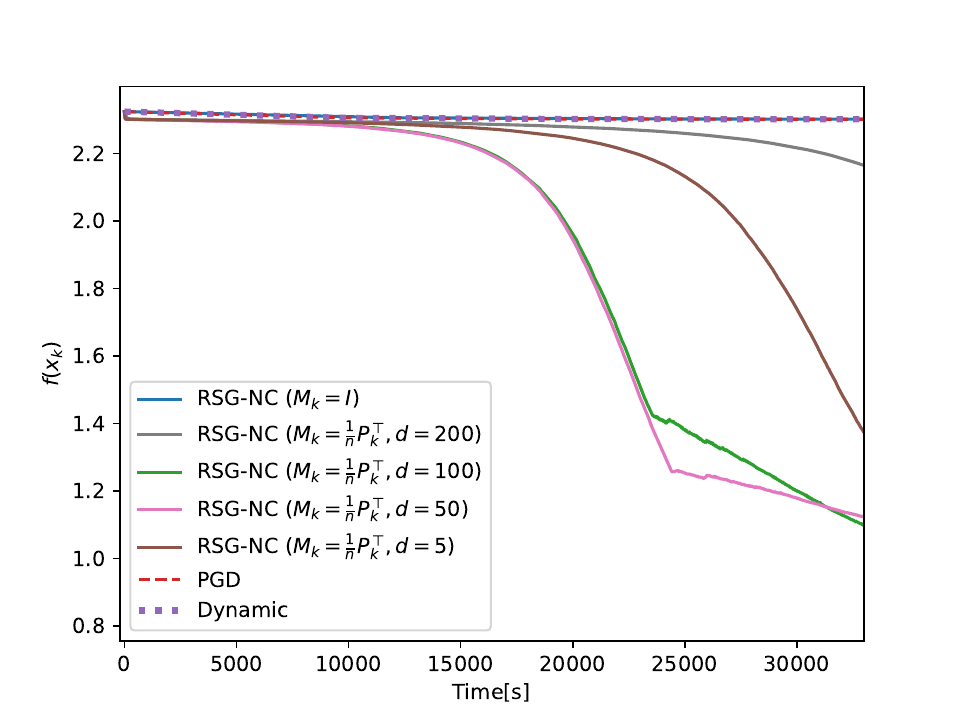}
	\caption{CNN with $l_2$-regularizer}
	\label{figure:Numerical Experiment4}
\end{figure} 

\section{Conclusion}
We proposed new methods combining the random projection and gradient projection method. We proved that they globally converge under linear constraints and nonlinear constraints.
When the original dimension is large enough, they converge in $O(\varepsilon^{-2})$. The numerical experiments showed some advantages of random projections as follows. 
First, our methods with randomness have the potential to obtain better solutions than those of their deterministic versions. Second, under non-smooth constraints, they did not become trapped at the boundary, whereas their deterministic versions did become trapped.
Last, our methods performed well when the gradients could not be obtained directly. In the future, we would like to investigate the convergence rate of our
algorithms in a non-smooth constraints setting.

\section{Compliance with Ethical Standards}
This work was partially supported by JSPS KAKENHI (23H03351) and JST ERATO (JPMJER1903).
There is no conflict of interest in writing the paper.

\bibliography{main}


\end{document}